\documentclass[10pt]{amsart}
\usepackage[OT1]{fontenc}
\usepackage{amsthm,amsmath}
\usepackage[numbers]{natbib}
\usepackage[colorlinks,citecolor=blue,urlcolor=blue]{hyperref}
\usepackage{array}
\usepackage{hhline}
\usepackage[table,xcdraw]{xcolor}
\usepackage{multirow}
\usepackage[toc,page,titletoc]{appendix}
\usepackage{xr-hyper}
\usepackage{hyperref}
\usepackage[british]{babel}
\usepackage{nameref}
\usepackage{amsfonts}
\usepackage{amssymb}
\usepackage{amsfonts}
\usepackage{mathrsfs}
\usepackage{amsfonts}
\usepackage{amstext}
\usepackage{amsopn}
\usepackage{pgfplots,geometry}
\usepackage{color}
\usepackage[]{subfig}
\usepackage[justification=justified]{caption}
\usepackage{tikz}
\usepackage{tikz-cd}
\usepackage{gensymb}
\usepackage[version=4]{mhchem}
\usepackage{graphicx}
\usepackage{chemfig}
\usepackage{url}
\usepackage[all]{xy}
\usepackage{lmodern}
\usepackage{pifont}
\usepackage[abs]{overpic}
\usepackage{texdraw}
\usepackage{extarrows}
\usepackage{bbm}
\usepackage[utf8]{inputenc}
\usepackage{mathtools}
\usepackage{enumitem}
\newcounter{tmp}
\allowdisplaybreaks
\usetikzlibrary{arrows, decorations.markings}
\usetikzlibrary{positioning}
\usetikzlibrary{fit}
\usetikzlibrary{backgrounds}
\usetikzlibrary{calc}
\usetikzlibrary{shapes}
\usetikzlibrary{mindmap}
\usetikzlibrary{decorations.text}
\pgfplotsset{compat=newest}

\numberwithin{equation}{section}
\theoremstyle{plain}
\newtheorem{theorem}{Theorem}[section]
\newtheorem{lemma}[theorem]{Lemma}
\newtheorem{corollary}[theorem]{Corollary}
\newtheorem{proposition}[theorem]{Proposition}

\newtheorem{conjecture}[theorem]{Conjecture}

\theoremstyle{definition}
\newtheorem{definition}[theorem]{Definition}
\newtheorem{example}[theorem]{Example}
\theoremstyle{remark}
\newtheorem{remark}[theorem]{Remark}

\newcommand{\cC}{\mathcal{C}}

\newcommand{\cE}{\mathcal{R}}

\newcommand{\cI}{\mathcal{I}}

\newcommand{\cL}{\mathcal{L}}

\newcommand{\cR}{\mathcal{R}}
\newcommand{\cS}{\mathcal{S}}
\newcommand{\cV}{\mathcal{C}}

\newcommand{\supp}{{\rm supp\,}}

\newcommand{\srcsize}{\@setfontsize{\srcsize}{5pt}{5pt}}

\newcommand\blue[1]{\textcolor{blue}{#1}}
\newcommand\red[1]{\textcolor{red}{#1}}

\newcommand\sbullet[1][.5]{\mathbin{\vcenter{\hbox{\scalebox{#1}{$\bullet$}}}}}

\makeatletter
\DeclareRobustCommand\widecheck[1]{{\mathpalette\@widecheck{#1}}}
\def\@widecheck#1#2{%
    \setbox\z@\hbox{\m@th$#1#2$}%
    \setbox\tw@\hbox{\m@th$#1%
       \widehat{%
          \vrule\@width\z@\@height\ht\z@
          \vrule\@height\z@\@width\wd\z@}$}%
    \dp\tw@-\ht\z@
    \@tempdima\ht\z@ \advance\@tempdima2\ht\tw@ \divide\@tempdima\thr@@
    \setbox\tw@\hbox{%
       \raise\@tempdima\hbox{\scalebox{1}[-1]{\lower\@tempdima\box
\tw@}}}%
    {\ooalign{\box\tw@ \cr \box\z@}}}
\makeatother

\ExplSyntaxOn
\keys_define:nn { mhchem }
 {
  arrow-min-length .code:n =
   \cs_set:Npn \__mhchem_arrow_options_minLength:n { {#1} } % default is 2em
 }
\ExplSyntaxOff
\mhchemoptions{arrow-min-length=1em}

\allowdisplaybreaks
\begin{document}
\title[Exponential ergodicity of endotactic SRSs]
  {Exponential ergodicity of  first order endotactic stochastic reaction systems}
\author{Chuang Xu}
%\address{Department of Mathematics\\
%University of Hawai'i at M\={a}noa, Honolulu, Hawai'i\\
%96822, USA
%}
\thanks{Department of Mathematics,
University of Hawai'i at M\={a}noa, Honolulu, Hawai'i,
96822, USA. Email: {chuangxu@hawaii.edu}. This work was supported by a start-up grant from the University of Hawai'i at M\={a}noa and a Travel Support for Mathematicians from the Simons Foundation (MP-TSM-00002379).}


\date{\today}

\noindent

\begin{abstract}
Chemical reaction networks are a widely accepted modeling framework for diverse science phenomena stemming from all disciplines of science, such as biochemistry, ecology, epidemiology, social and political science. In this paper we prove that every first order \emph{endotactic} stochastic mass-action reaction system (SMART) is essential (i.e., every state in the state space is within a closed communicating class of the underlying continuous time Markov chain model)  
and is exponentially ergodic. The proof is based on a recent result on first order endotactic reaction networks in a companion paper [C.X., First order endotactic reaction networks. arXiv:2409.01598v2]. Besides, we show that a stochastic reaction system (of possibly nonlinear propensities)  \emph{dominated} by a first order endotactic SMART is exponentially erogdic. To demonstrate the applicability of results, we provide various examples of higher order SMART, including e.g., (1) SMART with a first order endotactic \emph{asymptotic limit} as well as,  (2) joint of translations of first order endotactic SMART.  
\end{abstract}


\maketitle

%Color use rules: 
%\begin{itemize}
%    \item Red colored texts: Things to be fixed, likely wrongly stated.
%    \item Blue colored texts: Texts to be polished, and likely not accurately stated.
%    \item Texts in either  color not belonging to this paper (e.g., texts giving instructions to myself) are to remind myself to do the tasks later.
%\end{itemize}

\section{Introduction}

\begingroup
\setcounter{tmp}{\value{theorem}}% store current value of theorem counter
\setcounter{theorem}{0} %assign desired value to theorem counter
\renewcommand\thetheorem{\Alph{theorem}}% locally redefine the representation of the theorem counter

\subsection*{Background} 
Reaction network is a framework that unifies diverse compartmental models \cite{JS93} in both deterministic and stochastic regimes. It is ubiquitously used in modeling phenomena from diverse science areas, including genetics \citep{B45}, systems biology \cite{G03,C15}, population processes \cite{QG21}, computer science \citep{SCWB08}, game theory \citep{VRDF14}, and social sciences \citep{DW05}. Studies of reaction networks have evolved over decades into a theory,  called Chemical Reaction Network Theory (CRNT) \cite{G03,F19}. 

Every reaction network can be represented by a directed graph, called a \emph{reaction graph}. Such a graph defines a dynamical system, called a \emph{reaction system}, which can be either an ordinary differential equation (ODE) accounting for the concentrations of species in a deterministic regime from a macroscopic scale, or a continuous time Markov chain (CTMC) for molecule counts of species in a stochastic regime from a microscopic scale. 
A reaction network modeled as a CTMC is called a \emph{stochastic reaction system} (SRS). 

Key mathematical concerns in CRNT include answering questions about qualitative dynamical properties of reaction systems based on properties of the associated reaction graphs. An arguably infamous open question is the so-called Global Attractor Conjecture (GAC) proposed as early as 1970s
\cite{HJ72}: \emph{Every complex-balanced mass-action system has a globally attracting positive equilibrium within each positive stoichiometric compatibility class}.

Endotacticity, a property of \emph{embedded} graphs (in the Euclidean space), was introduced by Craciun, Nazarov, and Pantea in \cite{CNP13} to study GAC by putting the conjecture in a broader context. In \cite{CNP13}, it is conjectured that \emph{every endotactic ($\kappa$-variable\footnote{``$\kappa$-variable''  means that the reaction rate constants are functions of time taking positive values in a compact subset of positive reals.}) %which means the reaction rate constants are functions of time taking positive values in a compact subset of positive reals
 mass-action system as an ODE is \emph{permanent}}, which means that there exists a compact attractor which is compactly embedded in  $\mathbb{R}^d_{+}$.

A stochastic analogue of GAC \cite{CNP13}  is proved in the stochastic setting, due to the elegant observation of super-polynomial tail of the explicit Poisson-product form \cite{ACK10} of stationary distributions for complex-balanced SRS:
 \medskip
 
\begin{theorem}\cite{ACKK18}\label{thm:CB}
    Every complex-balanced \textbf{S}tochastic \textbf{M}ass-\textbf{A}ction \textbf{R}eaction sys\textbf{T}em (SMART) is ergodic\footnote{The original conclusion pertains to positive recurrence which is equivalent to ergodicity for irreducible CTMCs and hence for essential SMART.}
\end{theorem}  %(confined to closed communicating classes).} \cite{ACKK18}. 
\medskip

A more general conjecture, despite having informally existed for roughly similar long time since the seminal works by Kurtz \cite{K70,K71,K76} in the 1970s, is formally posed relatively recently by Anderson and Kim \cite{AK18} in 2018 which, if true, generalizes Theorem~\ref{thm:CB}:
\begin{conjecture}\label{conj:positive-recurrence}
Every weakly reversible SMART is ergodic.
\end{conjecture}

Below we summarize \emph{partial} known results on the advances of this conjecture. For insightful counter-examples, the interested reader is referred to e.g., \cite{AM18,AL25}.

\begin{table}[h]
\begin{center}
\begin{tabular}{| p{4.6cm} | p{2.3cm}| p{1.5cm} | p{4.3cm} |} 
\hline
Class of SMART &  Results & \hspace{-.1cm}Reference & Main Approach\\ \hline
Bimolecular with other assumptions& exponential ergodicity (EE)\footnote{Here we do not specify the norm (either in total variation distance or $L^2$-norm) for exponential convergence.} & \cite{AK18,ACKN20,AK23} & tier sequence\\ \hline
Bimolecular, weakly reversible with a strongly connected reaction graph, and the set of complexes contain a multiple of a single species for every species & EE & \cite{ACK20}&tier sequence and a discrete embedded chain argument\\ \hline
One-species, weakly reversible & EE & \cite{H18} & Linear Lyapunov functions\\ \hline
One-dimensional, endotactic & EE  &\cite{WX20} & A criterion established in \cite{XHW23}\\ \hline
%Three similar 2-species examples & transience, null recurrence and ergodicity, respectively& \cite{AM18}& scaling approach + piecewise defined Lyapunov functions\\ \hline
Bimolecular, triangular, weakly reversible & EE & \cite{LR23} & scaling approach and an adapted Lyapunov drift criterion\\ \hline
Ergodic &EE & \cite{ACFK25} & A modified path method and Poincar\'{e} inequality \\  \hline
Ergodic with a strong tier-1 cycle & sub-EE & \cite{KK25} & A modified path method and tier sequence \\ \hline
Bimolecular, endotactic with a stationary distribution & ergodicity & \cite{X25b} & A linear Lyapunov function\\ \hline
\end{tabular}
\end{center}
   \caption{A non-exhaustive list of  known results on ergodicity of SMART}
    \label{tab:summary_of_known_results}
\end{table}
Despite that transience and recurrence can be determined by up to three parameters of a SMART \cite{WX20,XHW23} in one dimension, the stability properties of CTMCs of \emph{linear} transition rates even just in the context of two-dimensional birth-death processes can be rather intricate (see \cite{Kesten76,H80} for the incomplete while likely most complete classification), in contrast to the classical stability theory for linear ODEs or linear compartmental systems \cite{JS93}. It is noteworthy that sufficient conditions for transience and recurrence (indeed, exponential ergodicity) for monomolecular SMART have been lately established in terms of the Hurwitz stability properties of a coefficient matrix (i.e., the Jacobian matrix of the corresponding linear  deterministic reaction system) in \cite[Corollary~9, Corollary~20]{CHX25}. It is noteworthy that expression of  existence of a linear Lyapunov function for exponential ergodicity \emph{in terms of an inequality for propensity functions} in the context of SRS appeared in the literature of CRNT, e.g. \cite{GBK14}. Nevertheless, it is non-trivial to identify a class of reaction network structures that meet such a condition without carrying out spectral analysis of the associated coefficient matrix. In contrast, instead of giving \emph{matrix conditions}, this paper provides sufficient \emph{network conditions} for exponential ergodicity, where spectral analysis becomes avoidable: One can determine exponential ergodicity of SRS by only \emph{looking} at the associated reaction graph.

\subsection*{Overview of Main Results}

A main result of this paper is as follows.

\begin{theorem}\label{thm:essential_loose_version}
Every first order\footnote{In this paper, first order and monomolecular is interchangeably used to refer to reaction networks whose reactants consist of no more than one molecule of species. In other references, e.g., \cite{JH07}, monomolecular may mean in a \emph{narrower} sense that products are also of no more than one molecule.} endotactic SMART is essential and exponentially ergodic.
\end{theorem}

It is probably worth pointing out first that \emph{not} every first order SMART is so. The example below demonstrates that monomolecular SMART may even undergo \emph{bifurcations}. 

\begin{example}
Consider the following monomolecular SMART:

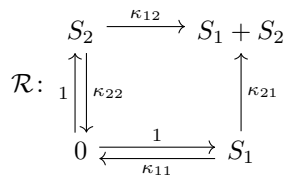
\[\mathcal{R}\colon 
\begin{tikzcd}[column sep=3em,row sep=3em]
S_2 \arrow[d,shift left=0.5ex,"\kappa_{22}"] \arrow[r,shift left=0.5ex,"\kappa_{12}"]& S_1+S_2\\
0\arrow[r,shift left=0.5ex,"1"]\arrow[u,shift left=0.5ex,"1"]&S_1 \arrow[l,shift left=0.5ex,"\kappa_{11}"]  \arrow[u,shift left=0ex, swap, "\kappa_{21}"]
\end{tikzcd}
\]
%\[\mathcal{R}\colon   S_2\ce{<=>[\kappa_{22}][1]}0\ce{<=>[1][\kappa_{11}]}S_1\ce{->[\kappa_{21}]}S_1+S_2\ce{<-[\kappa_{12}]}S_2,\] 
where $\kappa_{ij}>0$ for $i,j=1,2$. It is known from \cite[Theorem~4]{Kesten76} and \cite[Theorem~2]{H80} (c.f. \cite[Appendix I(B)]{H80}) that $\mathcal{R}$ modeled by an irreducible CTMC (indeed a birth-death process) on $\mathbb{N}^2_0$ is recurrent if and only if $$\kappa_{11}\kappa_{22}\ge\kappa_{12}\kappa_{21}$$ 
Moreover, we know $\mathcal{R}$ is exponentially ergodic if the inequality holds strictly \cite[Corollary~9]{CHX25}. Hence this SRS undergoes a \emph{bifurcation} (regarding recurrence) when $\kappa_{12}\kappa_{21}-\kappa_{11}\kappa_{22}$ crosses zero.
\end{example}

Theorem~\ref{thm:essential_loose_version} is a consequence of Theorem~\ref{thm:essential} and Theorem~\ref{thm:exp-ergodicity-endotactic}. The proof relies on a construction of a linear Lyapunov function, based on characterization (Proposition~\ref{prop:endotactic=weakly-reversible} and Proposition~\ref{prop:path-back-to-0}) of first order endotactic reaction networks in a companions work \cite{X25a}. As a consequence, \emph{every first order weakly reversible SMART is essential and \emph{exponentially} ergodic}, since weakly reversible reaction graphs are endotactic \cite{CNP13}.

It is noteworthy that every first order weakly reversible SMART is ergodic as a result of Theorem~\ref{thm:CB} since it is  complex-balanced \cite{ACK10}. Hence the result in this paper tells \emph{slightly more} than ergodicity for this class of reaction systems. 

Despite endotacticity is sufficient for exponential ergodicity for first order SMART, it (or even strong endotacticity) is in general \emph{insufficient} for ergodicity or even for recurrence or non-explosivity, as evidenced by various known examples   \cite{AM18,ACKN20}; even complex-balancing (which yields weak reversibility and hence endotacticity) does \emph{not} yield exponential-ergodicity \cite{ACFK25,KK25} as supported by the following popular example.
\begin{example}\label{ex:sub-exponential-ergodicity}
    The following second order reversible SMART is shown to be sub-exponentially ergodic in \cite{ACFK25} (see also \cite{KK25}), regardless of the positive reaction rate constants :
    \begin{equation}
        0\ce{<=>[][]} S_1+S_2\quad S_2\ce{<=>[][]} 2S_2
    \end{equation}
\end{example}

%Below we provide a typical example of a first order endotactic reaction network which is neither strongly endotactic nor weakly reversible. 
%\begin{example}\label{ex:introductory}
%Consider the following SMART 
%\cite{X25a} \begin{center}\begin{tikzcd}[row sep=1em, column sep = 1em]
%\mathcal{R}\colon 
% S_1\arrow[rightharpoonup]{r}{}&S_2\arrow[yshift=-.08cm,rightharpoonup]{l}{}&S_3\arrow[rightharpoonup]{r}{}&S_4\arrow[yshift=-.08cm,rightharpoonup]{l}{}&S_5\arrow{r}{}&S_6\arrow{r}{}&0\arrow{r}{}&2S_5+S_6
%\end{tikzcd}
%\end{center}
%This SMART is endotactic \cite{X25a} but not weakly reversible or strongly endotactic. Since  $\mathcal{R}$ does not  contain the full set of outflow reactions $S_i\ce{->[]}0$ for each $1\le i\le6$, exponential ergodicity does not follow from \cite{GBK14,AK23} either. 
%Nevertheless, it follows from Theorem~\ref{thm:essential_loose_version} that $\mathcal{R}$ is exponentially ergodic regardless of the positive reaction rate constants. 
%\end{example}

The loose statement of an extension of Theorem~\ref{thm:essential_loose_version} to SRS with \emph{nonlinear propensities} is given below.

\begin{theorem}\label{thm:nonlinear_loose_version}
    An SRS ``dominated'' by a first order endotactic SMART is exponentially ergodic on their common closed communicating classes.
\end{theorem}

For its precise statement, see Theorem~\ref{thm:exponential-ergodic-RN}; for a corollary with more applicable conditions, see Corollary~\ref{cor:exponential-ergodicity-for-nonlinear-models}. Below we provide an application of Theorem~\ref{thm:nonlinear_loose_version} to a higher order example.

%Below we provide a bimolecular SMART as an example, which seems \emph{atypical} (in the sense it does not seem to be covered) in  the prevalent references on ergodicity of bimolecular SMART listed in Table~\ref{tab:summary_of_known_results}.

%\noindent\textbf{Further Extensions.} The ergodicity results may hold beyond first order reaction systems. Below we provide several examples which are 
\begin{example}\label{ex-nonlinear}
Consider the following third order SMART \[\mathcal{R}\colon 
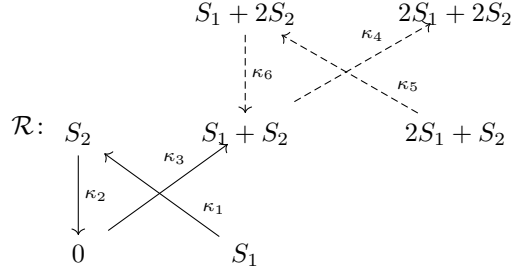
\begin{tikzcd}[column sep=3em,row sep=3em]
&S_1+2S_2 \arrow[d,shift left=0ex,dashed,"\kappa_6"] & 2S_1+2S_2\\
S_2 \arrow[d,shift left=0ex,"\kappa_2"] &S_1+S_2\arrow[ur,xshift =1ex,yshift =1ex, pos=.7, dashed, "\kappa_4"]& 2S_1+S_2\arrow[lu, pos=0.2, swap, dashed, "\kappa_5"]\\
0\arrow[ur,xshift =1ex,yshift =1ex, pos=.7, "\kappa_3"]&S_1\arrow[ul, pos=0.2, swap, "\kappa_1"]&
\end{tikzcd}
\] 
which can be represented as the joint of a first order endotactic SMART $\mathcal{R}_1$ consisting of reactions represented by solid arrows and a translation of $\mathcal{R}_1$ (consisting of the rest reactions in dashed arrows) disrespecting reaction rate constants. %, where $$\mathcal{R}_1\colon 
It follows from Theorem~\ref{thm:nonlinear_loose_version} that $\cR$ is exponentially ergodic independent of the rate constants (see Example~\ref{ex:copies-exp-erg} for more details).
\end{example}
Other examples of higher order SRS with a certain network pattern built upon a first order endotactic SMART include e.g., \emph{SRS with a first order endotactic SMART asymptotic limit} (Example~\ref{ex:bimolecular}).

\subsection*{Outline} 
Notation and basic concepts of reaction networks are introduced in Section~\ref{sect:preliminaries}. Proof of Theorem~\ref{thm:essential_loose_version}  is given in Section~\ref{sect:main_results}. %A general exponential ergodicity result of SRS with nonlinear propensity functions are given 
Proof of Theorem~\ref{thm:nonlinear_loose_version} is given in Section~\ref{sect:nonlinear}; various examples demonstrating the applicability of Theorem~\ref{thm:nonlinear_loose_version} is given in Section~\ref{sect:examples}. Brief discussion and outlooks are provided in Section~\ref{sect:discussion}. 
\endgroup
\section{Preliminaries}\label{sect:preliminaries}

\subsection*{Notation} Let $\mathbb{R}$ denote the set of real numbers, $\mathbb{R}_+\subseteq\mathbb{R}$ the non-negative reals, and $\mathbb{R}_{++}\subseteq\mathbb{R}$ the positive reals. Let $\mathbb{N}_0$ and $\mathbb{N}$ be the set of non-negative integers and that of positive integers, respectively. For $d\in\mathbb{N}$, let $[d]$ denote the set of consecutive integers from $1$ to $d$ $\{i\colon i=1,\ldots,d\}$.  For a vector $v=(v_1,\ldots,v_d)\in\mathbb{R}^d$ and for a set $V\subseteq\mathbb{R}^d$, let $\supp v\coloneqq\{j\colon v_j\neq0\}$ and $\supp V\coloneqq\cup_{v\in V}\supp v$ be be their \emph{support}, respectively. A vector $v\in\mathbb{R}^d$ is \emph{positive} and denoted $v>0$  if $v_i>0$ for all $i\in[d]$. A vector $v\in\mathbb{R}$ is \emph{negative} if $-v$ is positive. Let $\{e_i\}_{i=1}^d$ be the standard basis of $\mathbb{R}^d$, and $\mathbf{1}=\sum_{i=1}^de_i$ with the dependence on $d$ omitted which will be clear from the context. In contrast, \emph{$0$ is slightly abused for both a scalar and a vector}. For $x,y\in\mathbb{N}^d_0$, $x\ge y$ (component-wise), let $x^{\underline{y}}=\prod_{i=1}^d\prod_{\ell=0}^{y_i-1}(x_i-\ell)$ denote the \emph{descending factorial}. Let $\mathcal{M}_d(\mathbb{R})$ be the set of all $d$ by $d$ matrices with real entries. A matrix is \emph{Hurwitz stable} if all of its eigenvalues have negative real parts. A matrix is \emph{Metzler} if all of its off-diagonal entries are non-negative.

A \emph{reaction network} consists of three finite non-empty sets: the set of $d$ symbols, called \emph{species} $\mathcal{S}\coloneqq \{S_i\}_{i=1}^d$, the set of \emph{complexes} $\mathcal{C}\subseteq \{y=\sum_{i=1}^d\ell_iS_i\colon \ell_i\in\mathbb{N}_0\}$, and the set of \emph{reactions} $\mathcal{R}\subseteq \{y\to y'\colon y, y'\in\mathcal{C},\ y\neq y'\}$. Neglecting the symbols of species, every complex $y=\sum_{i=1}^d\ell_iS_i$ %representing binding of different (chemical) species
is identified as a vector $(\ell_1,\ldots,\ell_d)$ for convenience. Any reaction $y\to y'$ is indeed an ordered pair of complexes, where the complex $y$ is called the \emph{reactant} (or \emph{source}) and $y'$ the \emph{product}, and $y'-y$ is called the \emph{reaction vector}. By definition, \emph{all reaction vectors are non-zero}. A reaction network is represented by its \emph{reaction graph}\---a directed graph
$(\mathcal{C},\mathcal{R})$ with the set of vertices $\mathcal{C}$ and the set of directed edges $\mathcal{R}$. Moreover, let $\mathcal{C}_+$ denote the set of reactants. %For instance, there is a directed edge from $y$ to $y'$ in $(\cC,\cR)$ if and only if there exists a reaction $y\to y'$. 
Since all information of a reaction network is encoded in the set $\mathcal{R}$, we simply denote the reaction network by its set of reactions $\mathcal{R}$. 
A reaction network is \emph{weakly reversible} if its reaction graph consists of disjoint strongly connected components. A species $S_i$ is \emph{purely catalytic} if there is no molecule change of $S_i$ in any reaction. 
For the ease of exposition and without loss of generality, we assume throughout that \emph{no reaction network has purely catalytic species}. Otherwise, by convention, one can always \emph{embed} the kinetic effect of purely catalytic species in the edge weights of the reaction graph. The linear span of all reaction vectors of a reaction network $\cR$ in the real field is called the \emph{stoichiometric subspace}, denoted $\mathsf{S}_{\mathcal{R}}$. The dimension $\dim\mathsf{S}_{\cR}$ of $\mathsf{S}_{\mathcal{R}}$ is often referred to as the \emph{dimension} of the reaction network $\cR$ and the orthogonal complement of $\mathsf{S}_{\cR}$ in $\mathbb{R}^d$ is denoted by $\mathsf{S}_{\mathcal{R}}^{\perp}$.  A reaction network is called \emph{conservative} if $\mathsf{S}_{\mathcal{R}}^{\perp}$ contains a positive vector. Let $\ell_{\cR}$ be the number of strongly connected components of $(\cC,\cR)$. The \emph{deficiency} of $\cR$ is defined to be the integer $\#\cC-\dim\mathsf{S}_{\cR}-\ell_{\cR}$, which is always non-negative  \citep{F19}. For a weakly reversible reaction network, its deficiency is equal to the number of \emph{linearly independent} equations for the edge weights of the reaction graph to meet in order for the reaction network to be \emph{complex-balanced} \citep{J12,F19}. 
Given any reaction $y\ce{->}y'$, the $\ell_1$-norm $\|y\|_1$ of the reactant is called the \emph{molecularity} of the reaction; and $\max_{y\to y'\in\cR}\|y\|_1$ is the \emph{molecularity} of a reaction network $\cR$. Both a reaction and a reaction network having  molecularity one are called \emph{monomolecular}.  A reaction $y\to y'\in\mathcal{R}$ is called \emph{decreasing} if $y\le y'$.

A (stochastic) \emph{propensity function} $\lambda_{y\to y'}$ of a reaction $y\to y'\in\mathcal{R}$ is a nonnegative function defined on $\mathbb{N}^d_0$ which quantifies the likelihood that a reaction fires. The family $\Lambda\coloneqq\{\lambda_{y\to y'}\colon y\to y'\in\mathcal{R}\}$ of propensity functions of a reaction network $\mathcal{R}$ is called the \emph{stochastic kinetics} of $\mathcal{R}$; 
moreover, we call $(\mathcal{R},\Lambda)$ a \emph{stochastic reaction system} (SRS). We will simply use $\cR$ to denote an SRS whenever the stochastic kinetics is clear from the context. 
Given an SRS $(\mathcal{R},\Lambda)$. We say $(\mathcal{R}_*,\Lambda_*)$ is a \emph{sub-reaction system} of $(\mathcal{R},\Lambda)$ if  $\mathcal{R}_*$ consists a subset of reactions in $\mathcal{R}$ and $\Lambda_*\subseteq\Lambda$ consists of propensity functions of reactions in $\mathcal{R}_*$.

\begin{definition}\label{def:generic_kinetics}
    A stochastic kinetics $\Lambda=\{\lambda_{y\to y'}\colon y\to y'\in\mathcal{R}\}$ is called \emph{\textbf{G}eneric} if for every $y\ce{->}y'\in\mathcal{R}$,
\medskip

\noindent (\textbf{G}) 
$\lambda_{y\to y'}(x)>0\quad \iff\quad x\ge y,\quad \  x\in\mathbb{N}^d_0$
\end{definition}

Generic kinetics means that in order for a reaction in the reaction network to fire with a positive probability, there needs to be adequate molecules of species which can constitute the reactant.

 Almost all prevalent stochastic kinetics are generic \cite{AK11}. \textbf{Throughout this paper, we confine to generic SRS.} 

A particular generic stochastic kinetics $\Lambda$ is called \emph{stochastic mass-action kinetics}, 
 if the propensity function of each reaction $y\to y'$ is a product of falling factorials: $$\lambda_{y\to y'}(x)=\kappa_{y\to y'}x^{\underline{y}}$$ In this way, \emph{every SMART is associated with and can be defined by a weighted reaction graph} (with the edge weight being the reaction rate constant). For every SMART, its molecularity is usually referred to as its  \emph{order}. Hence a monomolecular SMART is of first order. 

Furthermore, we define the \emph{joint} of two SRS.

\begin{definition}
Let $(\mathcal{R}_1,\Lambda_1)$ and $(\mathcal{R}_2,\Lambda_2)$ be  two SRS with kinetics  
$\Lambda_i=\{\lambda^{(i)}_{y\to y'}\colon y\to y'\in\mathcal{R}_i\}$ for $i=1,2$. We define the \emph{joint} of two SRS  
as follows: $$\mathcal{R}_1\cup\mathcal{R}_2\coloneqq\{y\to y'\colon y\to y'\in\mathcal{R}_1\ \text{or}\ y\to y'\in\mathcal{R}_2\}$$ where its associated kinetics is given by
$$\lambda^{(1,2)}_{y\to y'}(x) = \begin{cases}\lambda^{(1)}_{y\to y'}(x) + \lambda^{(2)}_{y\to y'}(x),& \text{if}\ y\to y'\in\mathcal{R}_1\cap\mathcal{R}_2,\\
\lambda^{(1)}_{y\to y'}(x),& \text{if}\ y\to y'\in\mathcal{R}_1\setminus\mathcal{R}_2,\\
\lambda^{(2)}_{y\to y'}(x),& \text{if}\ y\to y'\in\mathcal{R}_2\setminus\mathcal{R}_1,
\end{cases}\quad \text{ for } x\in\mathbb{N}^d_0$$ 
\end{definition}

\subsection*{Modeling an SRS}
Given an SRS $(\mathcal{R},\Lambda)$ with $\Lambda=\{\lambda_{y\to y'}\colon y\to y'\in\mathcal{R}\}$, let $X_t$ denote the  counts of species of $\mathcal{R}$ at time $t\ge0$. %One can represent $X_t$ in terms of the weak solution to 
Then $X_t$ is a CTMC on the ambient state space $\mathbb{N}^d_0$ with its \emph{extended generator} $\mathcal{L}$ given by
\[\mathcal{L} f(x) = \sum_{y\to y'\in\mathcal{R}}\lambda_{y\to y'}(x) (f(x+y'-y) - f(x)),\]
for every 
$f\in\mathsf{D}(\mathcal{L})$, the set of all real-valued functions on $ \mathbb{N}^d_0$.

% Via a random time-change representation \cite{AK11}, $X_t$ solves the following stochastic differential equation \cite{AK11}:
%\begin{equation}\label{eq:CTMC-model-SRS}
%X_t=X_0+\sum_{y\to y'\in\mathcal{R}}Y_{y\to y'}\Bigl(\int_0^t\lambda_{y\to y'}(X_s)ds\Bigr)
%\end{equation}
%where $Y_{y\to y'}$ are independent unit Poisson processes indexed by reactions in $\mathcal{R}$ and $X_0$ is the (deterministic) initial species counts.

\begin{definition}
Given a generic SRS $(\mathcal{R},\Lambda)$, for two (possibly identical) states $z_1,z_2\in\mathbb{N}^d_0$, we say $z_1$ \emph{leads to} $z_2$ for $(\mathcal{R},\Lambda)$ or $z_2$ is \emph{reachable} from $z_1$ for $(\mathcal{R},\Lambda)$ and denoted by $z_1\rightharpoonup z_2$ if as states of the underlying CTMC, $z_1$ leads to $z_2$ \cite{N98}. In particular, if both $z_1\rightharpoonup z_2$ and $z_2\rightharpoonup z_1$ for $(\mathcal{R},\Lambda)$, then $z_1$ and $z_2$ \emph{communicate} for $(\mathcal{R},\Lambda)$ and is denoted by  $z_1\leftrightarrow z_2$. 
\end{definition}
By genericity of kinetics, $z_1\rightharpoonup z_2$ if and only if there exists a sequence of (possibly repeated) reactions $\{y_j\to y_j'\}_{j=1}^m\subseteq\mathcal{R}$ such that $z_2=z_1+\sum_{i=1}^m(y_i'-y_i)$ and
\[x+\sum_{i=1}^{j-1}(y_i'-y_i)\ge y_j,\quad \text{for}\ j=1,\ldots,m,\]
where by convention $\sum_{i=1}^{0}(y_i'-y_i)=0$.

The state space $\mathbb{N}^d_0$ can be  decomposed into communicating classes of different types \cite{XHW22}. An SRS is \emph{essential} if the ambient space $\mathbb{N}^d_0$ only consists of closed communicating classes.  It is known that \emph{every weakly reversible generic SRS is essential} \cite[Lemma~4.6]{PCK14} (see also \cite{XHW22}).

\begin{definition}\label{def:structural_equivalence}
Let $(\cR,\Lambda)$ and $(\cR',\Lambda')$  be two SRS sharing a common set of $d$ species. We say  $(\cR,\Lambda)$ can be \emph{structurally embedded} into $(\cR',\Lambda')$ and denoted by $(\cR,\Lambda)\hookrightarrow(\cR',\Lambda')$ if for every $x,y$ in the common ambient space $\mathbb{N}^d_0$, $x\rightharpoonup y$ for $(\cR,\Lambda)$ implies that  $x\rightharpoonup y$ for $(\cR',\Lambda')$.
Furthermore,  $(\cR,\Lambda)$ and $(\cR',\Lambda')$ are \emph{structurally equivalent} and denoted by $(\cR,\Lambda)\boldsymbol{\leftrightarrow}(\cR',\Lambda')$ if both $(\cR,\Lambda)\hookrightarrow(\cR',\Lambda')$ and $(\cR',\Lambda')\hookrightarrow(\cR,\Lambda)$. 
\end{definition}

\subsection*{Endotactic reaction networks}

Below, we recall endotactic and strongly endotactic reaction networks introduced in \cite{CNP13} and \cite{GMS14}, respectively.  Both concepts were originally introduced to study \emph{permanence} and \emph{persistence} of deterministic mass-action systems.

\begin{definition}
Let $\mathcal{R}$ be a reaction network. For every  $u\in\mathbb{R}^d\setminus\mathsf{S}_{\mathcal{R}}^{\perp}$, $u$ defines a partial order on $\mathbb{R}^d$: $$y\ge_u z \Leftrightarrow (y-z)\cdot u\ge0;\quad   y>_uz \Leftrightarrow (y-z)\cdot u>0$$ Let $\mathcal{R}_u\coloneqq\{y\to y'\in\mathcal{R}\colon (y'-y)\cdot u\neq0\}$ be the set of reactions whose reaction vectors are non-orthogonal to $u$. Let $\mathcal{C}_{u,+}$ be the set of reactants of the sub reaction network $\mathcal{R}_u$, and $\max_u\mathcal{C}_{u,+}$ be the set of $u$-maximal elements in $\mathcal{C}_{u,+}$.

Then $\mathcal{R}$ is \emph{endotactic} if for every given  $u\in\mathbb{R}^d\setminus\mathsf{S}_{\mathcal{R}}^{\perp}$, 
for every  reaction $y\to y'\in\mathcal{R}_u$ with a $u$-maximal reactant $y\in\max_u\mathcal{C}_{u,+}$, we have $y>_uy'$. Moreover, an endotactic reaction network is \emph{strongly endotactic} if for every $u\in\mathbb{R}^d\setminus\mathsf{S}_{\mathcal{R}}^{\perp}$, there exists $y\to y'\in\mathcal{R}_u$ with $y\in\max_u\mathcal{C}_+$ such that $y>_uy'$.
\end{definition}

 \section{First order endotactic reaction networks}\label{sect:main_results}

Below we focus on first order reaction networks. Given any first order reaction network $(\cC,\cR)$, let $(\cC^0,\cR^0)$ be the (possibly empty) weakly connected component of $(\cC,\cR)$ containing the zero complex and $(\cC^{\sbullet},\cR^{\sbullet})$ the complement of $(\cC^0,\cR^0)$. %with $$\cR^{\sbullet}=\cR\setminus\cR^0\coloneqq\{y\to y'\in\cR\colon y\to y'\notin\cR^0\}$$ 
Let $0\le d_0\le d$ be the number species in $\cR^0$. Based on the assumption that purely catalytic species are excluded in any reaction, $\cR^{\sbullet}$ may contain a \emph{proper} subset of species of $\cR$. %We recall some propositions of first order endotactic reaction networks established in a companion paper \cite{X25a}. These propositions  are crucial  for proving the main results of this paper.  

Given any SMART $\mathcal{R}$, 
let $A=(a_{ij})_{d\times d}\in\mathcal{M}_d(\mathbb{R})$ with $$a_{ij}=\sum_{e_i\to y'\in\mathcal{R}}\kappa_{e_i\to y'}(y_j'-y_j),\quad i,j=1,\ldots,d$$ be the \emph{net flow matrix} of $\mathcal{R}$
and $b=(b_1,\ldots,b_d)$ with 
\[b_i=\sum_{0\to y'\in\mathcal{R}}\kappa_{0\to y'}y_i',\quad i=1,\ldots,d\]  the \emph{constant inflow vector} of $\mathcal{R}$. Obviously 
$A$ is a Metzler matrix and $b\ge0$.

%Let $\cR$ be a first order SMART of $d$ species. Assume the following decomposition 
%\medskip

%\noindent \textbf{(HC)} $\cR^0$ of $d_0$ species  and $\cR^{\sbullet}$ of $d-d_0$ species share no common species, and $\cR^0$ has a \textbf{H}urwitz stable net flow matrix and $\cR^{\sbullet}$ is \textbf{C}onservative. 
%\medskip

%It is noteworthy that every first order endotactic SMART fulfills this assumption \cite[Theorem~5.2]{X25a}. 

%Hence $\cR^0 \colon S_1\ce{->}0$ has a Hurwitz stable matrix ($-\kappa_1$) while the underlying ambient subspace $\mathbb{N}_0$ consists of an absorbing state $0$ and the remaining states being single open communicating classes. Moreover, $\cR^{\sbullet}$, despite conservative, is \emph{not} weakly reversible.

\subsection{Essentialness}\label{subsect:structure}
In this subsection, for a first order SMART $\cR$, 
we address the following questions: 
\textbf{\begin{itemize}
    \item[(Q1)] Is the CTMC associated with $\cR^0$ always irreducible on $\mathbb{N}^{d_0}$?
    \item[(Q2)] Is $\cR^{\sbullet}$ weakly reversible and of deficiency zero (WRDZ)?
\end{itemize}}
 With affirmative answers to both \textbf{(Q1)} and \textbf{(Q2)}, one can easily show that $\cR$ is essential. 

Nevertheless, the answers to both questions can be negative in general, with a simple counter-example 
$$\cR\colon S_1\ce{->[$\kappa_1$]}0\quad S_2\ce{->[$\kappa_2$]}S_3$$

In the following, we will confirm the affirmative answers to both questions for every first order endotactic SMART, and hence show that \emph{every first order endotactic SMART is essential} (Theorem~\ref{thm:essential}). Before proceeding directly to the proof of Theorem~\ref{thm:essential}, we first  obtain some intuition from the example below.

\begin{example} Consider the following SMART
\[\cR\colon S_1\ce{->[\kappa_1]}S_2\ce{->[\kappa_2]}0\ce{->[\kappa_3]}2S_1\]
For convenience, we order reactions by the indices of their reaction rate constants. To show $\cR$ is essential, first note that $(0,0),\ (1,0),\ (0,1)$ communicate. This can be verified based on the observation of the path 
$(1,0)\rightharpoonup(0,1)\rightharpoonup(0,0)$ via reactions 1,2 and the path $(0,0)\rightharpoonup(2,0)\rightharpoonup(1,1)\rightharpoonup(1,0)$ via reactions 3,1,2. Then one can prove via repetitions of this cycle consisting of the three states $(0,0),\ (1,0),\ (0,1)$ that $\cR$ is essential due to the genericity of the kinetics. 
\end{example}

%Despite it seems intuitive that first order endotactic SRS are essential, \emph{not} all first order SRS are essential, e.g., see Example~\ref{ex-not-equivalent} below.

To prove Theorem~\ref{thm:essential}, we rely on a characterization (Proposition~\ref{prop:endotactic=weakly-reversible}) as well as a  connectivity property (Proposition~\ref{prop:path-back-to-0}) of first order endotactic reaction networks obtained in \cite{X25a}.

Let $\cR$ be a first order reaction network. Define its \emph{monomerization} $(\cC^{\spadesuit},\cE^{\spadesuit})$ by $$\cE^{\spadesuit}=\cE^*\cup\{0\to S_k\}_{k\in K}\quad \text{and}\quad \cC^{\spadesuit}=\{y,y'\colon y\ce{->[]} y'\in\cE^{\spadesuit}\},$$  where  $\mathcal{R}^* = \{y\to y'\in\cR\colon \|y\|_1=1\}$ consists of all monomolecular reactions in $\mathcal{R}$ and $K = \supp\{y'\colon 0\rightharpoonup y' \text{ for } \cR\}$.  Note that $\cR^{\spadesuit}$ is of first order and is of deficiency zero. 
%Assume \\

%\noindent\textbf{(WR)} $\cR^{\spadesuit}$ is weakly reversible. 
%\medskip

%Since $\cR^{\spadesuit}$ is of deficiency zero, it follows from \textbf{(WR)} that $\cR^{\spadesuit}$ is WRDZ.

\begin{proposition}\cite[Theorem~5.10]{X25a}\label{prop:endotactic=weakly-reversible}
Let $\cR$ be a  first order reaction network. Then
\[\cR\ \text{is endotactic}\ \Leftrightarrow \cR^{\spadesuit}\ \text{is endotactic}\ \Leftrightarrow \cR^{\spadesuit}\ \text{is WRDZ}\]
\end{proposition}

%Below is a connectivity property of first order endotactic reaction networks also serving as a crucial lemma for the proof of Theorem~\ref{thm:essential}.
\begin{proposition}\cite[Lemma~5.5]{X25a}\label{prop:path-back-to-0}
Let $(\cC,\cR)$ be first order endotactic reaction network of $d$ species. Then $(\cC^0,\cR^0)$ and $(\cC^{\sbullet},\cR^{\sbullet})$ share no species. %, and $\cR^{\sbullet}$ is conservative.
Assume  $\cR^0\neq\emptyset$. Let  $$J=\{j\in[d]\colon\ e_j\rightharpoonup 0\},\ K=\supp\{y'\in\cV^0 \colon 0\rightharpoonup y'\},\ L=\{\ell\in J\setminus K\colon e_{k}\rightharpoonup e_{\ell}, \forall k\in K\}$$ Then
\begin{equation}\label{eq:support_of_non-zero_sources}
K\neq\emptyset;\quad K\cup L= J = \supp\cV^0
\end{equation}

 In other words, for every $j\in[d]$, there exists a path from $e_j$ to $0$ in $\cR^0$ if and only if either there exists a path from $0$ to a complex $y'\in\cV^0$ with $y'_j>0$ or there exists a path from $0$ to a complex $y'\in\cV^0$ with $y'_{k}>0$ and there exists a path from $e_{k}$ to $e_j$.
 In particular, there exists a path from every non-zero reactant in $\cR^0$ to $0$. 
\end{proposition}

\begin{theorem}\label{thm:essential}
Let $\cR$ be a first order endotactic SMART. 
Then $\cR^{\sbullet}$ is WRDZ and the CTMC associated with $\cR^0$ is irreducible on $\mathbb{N}^{d_0}_0$. In particular, $\cR$ is essential.
\end{theorem}
\begin{proof}
First observe that from Proposition~\ref{prop:endotactic=weakly-reversible} $\cR^0$ and $\cR^{\sbullet}$ do not share species, and hence every communicating class 
 of $\cR$ is the Cartesian product of a communicating class of $\cR^0$ in $\mathbb{N}^{d_0}_0$ and a communicating class of $\cR^{\sbullet}$ in $\mathbb{N}^{d-d_0}_0$. Since \emph{weakly reversible generic SRS is essential} \cite[Lemma~4.6]{PCK14} (see also \cite{XHW22}), that $\cR$ is essential follows from (1) $\cR^{\sbullet}$ is WRDZ, and (2) the CTMC associated with $\cR^0$ is irreducible on $\mathbb{N}^{d_0}_0$.

Let $\cR^{\spadesuit}$ be the monomerization of $\cR$. By Proposition~\ref{prop:endotactic=weakly-reversible}, $\cR$ and $\cR^{\spadesuit}$ are both endotactic. Since $\cR^{\spadesuit}$ is weakly reversible, by the definition, $(\cR^{\spadesuit})^{\sbullet}=\cR^{\sbullet}$ and $(\cR^{\spadesuit})^0$ are both weakly reversible as well. Hence (1) follows from that \emph{every first order weakly reversible reaction network is of deficiency zero} \cite{ACK10} (see also \cite{X25a}). Since $\cR^0$ and $\cR^{\sbullet}$ do not share species, it suffices to show (2) under the assumption that $\cR=\cR^0$. To show (2) under this assumption, it further suffices to show (using the connectivity property in Proposition~\ref{prop:path-back-to-0}) that  $0\leftrightarrow e_j$ for $\cR$ for all $j\in[d]$.

%; moreover, it follows from Proposition~\ref{prop:path-back-to-0} that $\mathcal{R}^0$ and $\cR^{\sbullet}$ do not share species (neither do $(\cR^{\spadesuit})^0$ and $(\cR^{\spadesuit})^{\sbullet}$). In the light of that , we assume w.l.o.g. that $\cR=\cR^0$. Hence to prove the conclusion under this assumption, it suffices to show that  $0\leftrightarrow e_j$ holds for all $j\in[d]$.

By Proposition~\ref{prop:path-back-to-0}, $J= \supp\cC^0 = [d]$. 
It then remains to show $0\rightharpoonup e_j$ for $\cR$ for all $j\in[d]$. By Proposition~\ref{prop:path-back-to-0}, it further suffices to show $0\rightharpoonup e_j$ for $\cR$ for all $j\in K$ only. By the definition of $K$, there exists $0\to y'\in\cE$ such that either $y'_j>0$, or $y'_k>0$ and $e_k\rightharpoonup e_j$. Then it suffices to show $y'\rightharpoonup e_j$ in the former case and $y'\rightharpoonup e_k$ in the latter case. We only prove $y'\rightharpoonup e_j$ as the same arguments apply to the other case. To see this is true, since $\cR$ is generic, one can simply create a path from $y'$ to $e_j$ in terms of $y'_{\ell}$ repetitions of paths from $e_{\ell}$ to $0$ for $\ell\in\supp y'\setminus\{j\}$ and $y'_j-1$ repeated paths from $e_j$ to $0$, with the paths arranged in any (preferred) order.

%Conversely, for any two states $x,z\in\mathbb{N}^d_0$ with $x\rightharpoonup z$ for $\cR^{\spadesuit}$, we next will show $x\rightharpoonup z$ for $\mathcal{G}$. If $z$ can be reached from $x$ only by first order reactions, then in the light of that all first order reactions in $\cR^{\spadesuit}$ are also in $\mathcal{G}$ we know that $x\rightharpoonup z$ for $\mathcal{G}$. Hence assume w.l.o.g. that $z$ must be reached from $x$ via some zeroth order reaction $0\ce{->}S_k$ in $\cR^{\spadesuit}$. \blue{STOPPED HERE.}

%It suffices to show $0\leftrightarrow e_j$ for all $j\in[d]$.  
\end{proof}
\begin{remark}\label{re:structure}

    %\begin{itemize}
     %   \item From the proof of Theorem~\ref{thm:essential}, the projections of all communicating classes onto the set of species of $\cR^0$ are  equal to $\mathbb{N}^{d_0}_0$. In contrast, the projections of all communicating classes onto the set of species of $\cR^{\sbullet}$ are Cartesian products of finite closed communicating classes. 
        %\item Despite \textbf{(HC)} is insufficient for affirmative answers to \textbf{(Q1)} and \textbf{(Q2)}, \textbf{(HC)} and \textbf{(WR)} together are sufficient.
        %\item 
        In order for  Theorem~\ref{thm:essential} to hold, it is easy to observe from the proof that stochastic mass-action kinetics is \emph{not} necessary; instead, genericity of kinetics of $\cR$ (and $\cR^0$) suffices. 
        %\end{itemize}
\end{remark}

Out of independent interest, one can easily show that $\cR\boldsymbol{\leftrightarrow}\cR^{\spadesuit}$. Furthermore,  $\cR\hookrightarrow\cR^{\spadesuit}$ holds \emph{regardless of the endotacticity of $\cR$} (Proposition~\ref{prop:embedding}); nevertheless, the converse embedding $\cR^{\spadesuit}\hookrightarrow\cR$ (and hence $\cR^{\spadesuit}\boldsymbol{\leftrightarrow}\cR$) may fail if $\cR$ is \emph{not} endotactic (and thus neither is $\cR^{\spadesuit}$ by Proposition~\ref{prop:endotactic=weakly-reversible}).
\begin{example}\label{ex-not-equivalent}
Consider the following generic first order SRS  $$\cR \colon 0\ce{->[$\kappa_1$]} S_1+S_2$$
Note that $$\cR^{\spadesuit}\colon S_2\ce{<-[$\kappa_1$]}0\ce{->[$\kappa_1$]} S_1$$
Hence $0 \rightharpoonup e_1$ for $\cR^{\spadesuit}$ but $0 \not\rightharpoonup e_1$ for $\cR$. %This example shows that for a generic first order SRS $\cR$, $\cR$ and $\cR^{\spadesuit}$ in general may not be structurally equivalent. %\red{Check if we  eventually define structural equivalence. If not, emphasize the term and cite \cite{XHW22}.}
\end{example}

Essentialness of endotactic SMART is a property peculiar to \emph{first} order endotactic SMART.
\begin{example}
Consider the following second order SMART
\[\cR\colon 0\ce{->[$\kappa_1$]} S_2\quad 2S_2\ce{->[$\kappa_2$]} S_1+S_2\quad 2S_1\ce{->[$\kappa_3$]} S_1\]
It is readily verified that $\cR$ is endotactic while is \emph{not} essential. Indeed, $(0,1)$ is an absorbing state, %(in particular, it is a so-called ``trapping state'' in \cite{XHW22}), 
$\mathbb{N}^2$ is the other closed communicating class, and all other states are open singleton classes. %(so-called ``escaping states'' in \cite{XHW22}). 
It is noteworthy that $\cR$ does have a second order WRDZ realization (in place of the monomerization for first order reaction networks) 
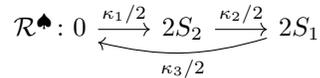
\begin{center}\begin{tikzcd}[row sep=2em, column sep = 2em]
\cR^{\spadesuit}\colon 0\rar{}{\kappa_1/2} &2S_2\rar{}{\kappa_2/2}&2S_1\arrow[bend left=15]{ll}{\kappa_3/2}
\end{tikzcd}
\end{center}
%$\widehat{\cR}\colon 0\to 2S_2\to 2S_1\to 0$ 
%(indeed, it is a ``\emph{strong} realization'' \cite{X25a}), 
which as a generic SRS is essential. %This example also illustrates that realization of a reaction network, despite is dynamically equivalent to the original reaction network in the deterministic sense, may even not be structurally equivalent to the latter.  
\end{example}
%\begin{example}
%Consider the first order reaction network $0\ce{->}S$. It is easy to see that all states in the state space $\mathbb{N}_0$ are transient.
%\end{example}

%\blue{Is it possible to construct a CTMC of affine transition rates which are sub-exponentially ergodic?}

Obviously, endotacticity is \emph{unnecessary} for essentialness of first order SRS. %To break endotacticity while maintaining essentialness of the SRS, one can simply add first order endotacticity-violating reactions to a first order endotactic SRS.

\begin{example}\label{ex:endotacticity-is-unnecessary}
Consider the first order SMART $$0\ce{<=>[1][2]}S_1\ce{->[1]}2S_1$$
is \emph{not} endotactic while is structurally equivalent to an endotactic SMART $0\ce{<=>[1][1]}S_1$, and hence is also essential. 
\end{example}

\subsection{Exponential ergodicity}\label{subsect:ergodicity}

Given any SRS, the  %transition rate 
$Q$-matrix of the underlying CTMC modeling this SRS %defined in \eqref{eq:CTMC-model-SRS} 
uniquely decomposes the ambient state space $\mathbb{N}^d_0$ into communicating classes \cite{N98} (see also \cite{XHW22}). We speak of a dynamical property of an SRS by that of the underlying CTMC.
\begin{definition}
We say an SRS is \emph{recurrent}/\emph{(exponentially) ergodic} on a closed communicating class if the underlying CTMC  on that class is so. An SRS is \emph{transient} if all states in the ambient space are transient for the underlying CTMC. In particular, an essential SRS is \emph{recurrent}/\emph{(exponentially) ergodic} if it is so on every closed communicating class. %Hence an SRS is recurrent if $\mathbb{N}^d_0$ consists of only finite closed  communicating classes.
\end{definition}

Thanks to Theorem~\ref{thm:essential}, we know the ambient state space $\mathbb{N}^d_0$ of any first order endotactic SRS is decomposed into closed irreducible components.

%\begin{example}\label{ex:revisit-introductory}
%\[\cR\colon S_1+S_2\ce{<=>[][]}S_2+S_3,\quad 2S_1+S_2\to S_1+2S_2\to S_1+S_2\to 2S_1+2S_2\]
%Note that this is a binary endotactic RN. \blue{Find a weakly reversible example.} 
%Consider the following first order MAK SRS/Recall Example~\ref{ex:introductory} in the Introduction:

%However, it is not double full.  Hence none of the previous results in \cite{AK18,ACK20,ACKN20,AK23} apply. %Hence the results in \cite{AK18} do not apply.
%\end{example}

%For a given reaction network $\cR$, let $\cR^0$ be its (possbily empty) weakly connected component containing the zero complex and $\cR^{\bullet}=\cR\setminus\cR^0$. The following proposition will be used to prove exponential ergodicity of first order endotactic  SMART. 
%\begin{proposition}\label{prop:1-endotacticity}\cite[Proposition~6.1]{X25a}
%Let $\cR$ be a first order $\mathbf{1}$-endotactic SMART. Let $A$ be its net flow matrix. Then $A=\text{diag}\{A_0,A_1\}$, where $A_0$ is Hurwitz stable with $w_0\in\mathbb{R}^{d_0}_{++}$ such that $A_0w_0^T<0$ and $A_1$ is a Laplacian matrix of the  weakly reversible reaction subgraph $\mathcal{R}^{\bullet}$. %If $A$ is non-singular, 
%Then there exists $w\in\mathbb{R}^d_{++}$ such that $Aw^T<0$ if and only if $\cR=\cR^0$.
%\end{proposition}

\begin{theorem}\label{thm:exp-ergodicity-endotactic}
 
A first order SMART is exponentially ergodic if it is 
\begin{enumerate}
\item[(i)] endotactic, or 
\item[(ii)] strongly endotactic, or 
\item[(iii)] weakly reversible.
\end{enumerate}
\end{theorem}
\begin{proof}
    Below we only prove case (i). Then cases (ii) and (iii) follow from (i) since both strongly endotactic reaction networks and weakly reversible reaction networks are endotactic \cite{CNP13} (see also \cite{X25a}).  

If $\cR^0=\emptyset$, then $\cR$ is conservative since $\mathbf{1}\in\mathsf{S}_{\cR}^{\perp}$ and hence exponential ergodicity is a consequence of the finite state irreducible underlying CTMC \cite[Theorem~5.3]{J13}. Hence we assume w.l.o.g. that $\cR^0\neq\emptyset$. By \cite[Proposition~6.1]{X25a}, the net flow matrix $A_0$ associated with $\cR^0$ is Hurwitz stable. Since $A_0$ is Metzler, $A_0$ is non-singular and $A_0^{-1}$ is non-negative, and there exists a positive column vector $v_0\in\mathbb{R}^{d_0}$ such that $A_0v_0<0$ \cite[Chapter 6, Theorem~2.3]{BP94} (c.f. \cite[Theorem~2.5.3]{HJ94}). It is easy to verify that $(v_0,\mathbf{1}) \cdot x$ is a Lyapunov function for exponential ergodicity of $\cR$ by Proposition~\ref{prop:drift_criterion_for_exp-ergodicity}, where $\mathbf{1}\in\mathbb{R}^{d-d_0}$. 
\end{proof}

\begin{remark}
\begin{itemize}
    \item Theorem~\ref{thm:exp-ergodicity-endotactic} fails in general for higher order SMART, e.g., see Example~\ref{ex:sub-exponential-ergodicity}. %In  \cite{KK25},  it is shown that the following weakly reversible SMART is strictly \emph{sub}-exponentially ergodic, disrespecting the reaction rate constants
%\[0\ce{<=>[]}S_1+S_2,\quad S_2\ce{<=>}2S_2\] 
It follows from \cite{WX20}  that every 1-dimensional endotactic SMART is  exponentially ergodic. In particular, every 1-dimensional weakly reversible SMART is exponentially ergodic. Hence in the light of Corollary~\ref{cor:exponential-ergodicity-for-nonlinear-models}(iii), the maximal \emph{order} and the maximal \emph{dimension} for weakly reversible SMART to be exponentially ergodic (in a \emph{universal} sense, i.e., every such SMART is so) are both equal to one. 
    \item It follows from the proof of \cite[Proposition~6.1]{X25a} that $A_0\mathbf{1}\le0$. Nevertheless,   $V(x)=\mathbf{1}\cdot x=\sum_{i=1}^dx_i$ {may \emph{not} always be} a Lyapunov function for exponential ergodicity of a first order endotactic SMART.  For instance, consider $$\cR\colon S_1\ce{->[\kappa_{1}]}S_2\ce{->[\kappa_{2}]}0\ce{->[\kappa_{0}]}S_1+S_2,$$ where $\cR=\cR^0$ and  $A_0\mathbf{1}=(0,-\kappa_{2})^T\not<0$. Below, we \textbf{provide a choice for $\boldsymbol{v_0}$} in the proof of Theorem~\ref{thm:exp-ergodicity-endotactic}  by construction. Let $\mathbf{I}$ be the $d_0$ by $d_0$ square matrix of all entries being 1. Since $A_0$ is Metlzer and Hurwitz, $\tfrac{1}{2}\mathbf{I}-\epsilon(A_0^{-1})$ is nonnegative for some  small $\epsilon>0$ \cite[Chapter 6, Theorem~2.3]{BP94} (c.f. \cite[Theorem~2.5.3]{HJ94}). Let  $v_0=(\mathbf{I}-\epsilon A_0^{-1})\mathbf{1}$.  Then $v_0\ge\tfrac{1}{2}\mathbf{I}\mathbf{1}=\tfrac{1}{2}d_0\mathbf{1}>0$; moreover, $$A_0v_0=A_0(\mathbf{I}-\epsilon A_0^{-1})\mathbf{1}=d_0A_0\mathbf{1}-\epsilon A_0A_0^{-1}\mathbf{1}\le-\epsilon\mathbf{1}<0$$
    \item Nonetheless, endotacticity is \emph{unnecessary} for exponential ergodicity of first order SRS, e.g., Example~\ref{ex:endotacticity-is-unnecessary}.
\end{itemize}
\end{remark}

%\red{Remove later.} Despite it is always non-explosive \cite{X25b}, a first order $\mathbf{1}$-endotactic SMART can be transient.

%\begin{example}
%Consider 
%\[\cR\colon 0\ce{<-[$\kappa_1$]}S_1\ce{->[$\kappa_2$]}S_2\]
%It is readily verified that $\cR=\cR^0$ is $(1,1)$-endotactic but \emph{not} endotactic. It is straightforward to verify that it neither is essential nor has a non-singleton closed class and hence is transient on any non-singleton class.
%\end{example}

\section{Exponential ergodicity of SRS of \emph{nonlinear} propensity functions}\label{sect:nonlinear}

%Now we slightly extend Corollary~\ref{cor:endotactic-is-ergodic} to obtain a sufficient condition for exponential ergodicity of reaction systems of superlinear propensity functions.

%In this section, we provide various examples and sufficient conditions for an SRS to be exponentially ergodic.

%\subsection{SRS dominated by a first order endotactic SMART}

In this section, we first show that every SRS ``dominated'' (in the sense of one of the conditions of \eqref{eq:condition-1}-\eqref{eq:condition-4} below) by a first order endotactic SMART is exponentially ergodic.

The following \textbf{S}ublinear \textbf{U}pper \textbf{B}ound assumption technically ensures the existence of a linear Lyapunov function for exponential ergodicity.
\medskip

\noindent$(\textbf{SUB})$
$(\cR,\Lambda)$ with $\Lambda=\{\lambda_{y\to y'}\colon y\to y'\in\cR\}$, is a generic SRS with sub reaction system $(\widehat{\cR},\widehat{\Lambda})\subseteq(\cR,\Lambda)$ which is a first order endotactic SMART, and there there exists a positive column vector $v_0\in\mathbb{R}^{d_0}$ satisfying $A_0v_0<0$ and a \emph{strictly sub-linear}\footnote{A function $f\colon \mathbb{N}_0^{d}\to\mathbb{R}$ is \emph{strictly sub-linear} if $\lim_{\|x\|_1\to\infty}\tfrac{f(x)}{\|x\|_1}=0$.} function $f\colon \mathbb{N}^d_0\to\mathbb{R}$  such that for all but finitely many $x\in\mathbb{N}_0^d$, 
\begin{equation}\label{eq:condition-1}
\sum_{y\ce{->}y'\in\cR\setminus(\widehat{\cR})^0}\lambda_{y\to y'}(x)((v_0,\mathbf{1})\cdot(y'-y))\le f(x),
\end{equation} 
where $A_0$ is the net flow matrix of $(\widehat{\cR})^0$ (the weakly connected component of $\widehat{\cR}$ containing the zero complex), $\mathbf{1}\in\mathbb{R}^{d-d_0}$, and the left hand side of \eqref{eq:condition-1} becomes zero if $\cR=(\widehat{\cR})^0$.

\begin{theorem}\label{thm:exponential-ergodic-RN}
Let $\cR$ be an SRS fulfilling (\textbf{SUB}) with a first order endotactic SMART $\widehat{\cR}$. Assume $\Gamma$ is a common closed communicating class for $\cR$ and $\widehat{\cR}$. %Assume ($\mathbf{H1}_{\mathcal{G}_1}$),  ($\mathbf{H2}_{\mathcal{G}_1,\mathcal{G}_2}$).% and ($\mathbf{H3}_{\mathcal{R}_3}$).
 Then $\cR$ is %$\mathcal{G}_1$ and $\mathcal{G}_2$ are all 
 exponentially ergodic on $\Gamma$. %on %every closed communicating class of 
% the ambient space $\mathbb{N}^d_0$.
%Let $(\mathcal{G}_i,\Lambda_i)$ for $i=1,2,3$ be three SRSs.  %Assume ($\mathbf{H1}_{\mathcal{G}_1}$),  ($\mathbf{H2}_{\mathcal{G}_1,\mathcal{G}_2}$) and ($\mathbf{H3}_{\mathcal{R}_3}$). Then $\mathcal{G}_1$, $\mathcal{G}_2$, $\mathcal{G}_2\cup\mathcal{R}_3$ are all exponentially ergodic on each closed irreducible component of the ambient space $\mathbb{N}^d_0$.
\end{theorem}

\begin{proof}
%The proof is of similar spirit of \cite{GBK14}. We provide a proof for the reader's convenience. 
%It suffices to prove exponential ergodicity  for $\mathcal{G}_2$. 
 The proof builds upon that $(v_0,\mathbf{1})\cdot x$ is a linear Lyapunov function for exponential ergodicity, %based on the characterization of first order endotactic reaction networks \cite{X25a} 
 as used in the proof of Theorem~\ref{thm:exp-ergodicity-endotactic}. 

%By Lemma~\ref{le:disjointness}, it suffices to assume $\cR^{\sbullet}=\emptyset$, since CTMCs on all finite closed communicating classes are exponentially ergodic \cite[Theorem~5.3]{J13}. 

%Since $\Gamma$ is a closed communicating class of $\widehat{\cR}$, we have  $\Gamma=\Gamma_1\times\Gamma_2$ with $\Gamma_2\in\mathbb{N}^{d-d_0}_0$ being finite and $d_0\le d$. 
%Hence for the same sake as in the proof of Theorem~\ref{thm:exp-ergodicity-endotactic}, the underlying CTMC associated with $(\widehat{\cR})^{\sbullet}=\widehat{\cR}\setminus(\widehat{\cR})^0$ is irreducible on a finite state space, and thus is exponentially ergodic. We therefore assume w.l.o.g. that $d_0=d$ so that $\Gamma=\mathbb{N}^d_0$ is the unique closed communicating class of the CTMC associated with $\cR$.

Since $\Gamma$ is a closed communicating class of $\widehat{\cR}$, let   $\Gamma=\Gamma_1\times\Gamma_2$ with $\Gamma_2\in\mathbb{N}^{d-d_0}_0$ being (possibly empty and hence also) finite and $d_0\le d$. Let $\mathcal{L}_{\cR}$ and $\mathcal{L}_{(\widehat{\cR})^0}$ be the extended generators of $\mathcal{R}$  and $(\widehat{\cR})^0$, respectively. It follows from (\textbf{SUB}) that there exists a constant $c>0$ such that for all but finitely many $x=(x^{(1)},x^{(2)})\in\Gamma_1\times \Gamma_2$, 
\begin{align*}
\mathcal{L}_{\mathcal{R}}(v\cdot x) = & \mathcal{L}_{(\widehat{\cR})^0} (v_0\cdot x^{(1)}+\|x^{(2)}\|_1) + \sum_{y\to y'\in \mathcal{R}\setminus(\widehat{\cR})^0}\lambda_{y\to y'}(x)((v_0,\mathbf{1})\cdot(y'-y))\\
=& \mathcal{L}_{(\widehat{\cR})^0} v_0\cdot x^{(1)} + \sum_{y\to y'\in \mathcal{R}\setminus(\widehat{\cR})^0}\lambda_{y\to y'}(x)((v_0,\mathbf{1})\cdot(y'-y))\\
\le & (A_0v_0)\cdot x^{(1)} + b_0\cdot v_0+ f(x)\le -c (v_0\cdot x), 
\end{align*} 
since $f$ is strictly sublinear and $A_0v_0<0$, where $A_0$ and $b_0$ are the net flow matrix and the constant inflow vector of $(\widehat{\cR})^0$, respectively. By Proposition~\ref{prop:drift_criterion_for_exp-ergodicity},  
$\mathcal{R}$ is  exponentially ergodic.
\end{proof}

\begin{corollary}\label{cor:exponential-ergodicity-for-nonlinear-models}
Let $(\cR,\Lambda)$ with $\Lambda=\{\lambda_{y\to y'}\colon y\to y'\in\cR\}$ be a generic SRS containing  a first order endotactic SMART sub reaction system $(\widehat{\cR},\widehat{\Lambda})\subseteq(\cR,\Lambda)$. %$\widehat{\cR}$ be its sub reaction network consisting of all zero-th or first order reactions in the weakly connected component of $\cR$ containing the zero complex. Assume $(\widehat{\cR},\widehat{\Lambda})$ with $\widehat{\Lambda}\coloneqq\{\lambda_{y\to y'}\colon y\to y'\in\widehat{\cR}\}$ is a first order endotactic SMART. 
Let $A_0$ be the net flow matrix of $(\widehat{\cR})^0$ with $v=(v_0,\mathbf{1})$ be a positive vector with $A_0v_0<0$. Let $\Gamma$ be a common closed communicating class of $\cR$ and $\widehat{\cR}$. Then $(\mathcal{R},\Lambda)$ is exponentially ergodic on $\Gamma$ if  one of the following conditions holds: 
\begin{enumerate} 
\item[(i)] for every reaction $y\to y'$ in $\mathcal{R}\setminus (\widehat{\cR})^0$, 
there exists a constant $C>0$ such that for all but finitely many $x\in\Gamma$, we have
\begin{equation}\label{eq:condition-2}
\lambda_{y\to y'}(x)(v\cdot(y'-y))\le C
\end{equation}
\item[(ii)] every reaction in $\mathcal{R}\setminus (\widehat{\cR})^0$ is decreasing along $v$: for every $y\ce{->}y'\in\mathcal{R}\setminus (\widehat{\cR})^0$,
\begin{equation}\label{eq:condition-3}
v\cdot(y'-y)\le0
\end{equation}
\item[(iii)] every reaction in $\mathcal{R}\setminus (\widehat{\cR})^0$ is decreasing:  for every $y\ce{->}y'\in\mathcal{R}\setminus (\widehat{\cR})^0$,
\begin{equation}\label{eq:condition-4}
y'\le y
\end{equation}
\end{enumerate}  
\end{corollary}

\begin{proof} It is readily verified that $\text{(iii)}\implies \text{(ii)}\implies \text{(i)}$, and (i) implies (\textbf{SUB}), and hence the conclusion follows from Theorem~\ref{thm:exponential-ergodic-RN}. 
\end{proof}

%It is crucial to assume $\widehat{\cR}$ is a sub reaction network of $\cR$ in the assumption (\textbf{SUB}). Indeed, Theorem~\ref{thm:exponential-ergodic-RN} may fail even if there exists a first order endotactic SMART $\widehat{\cR}\not\subseteq\cR$ which is structurally equivalent to $\cR$.

%\begin{example}
%Let $$\cR\colon 0\ce{->}S_1$$ and $\widehat{R}\colon 0\ce{<=>[\kappa_1][\kappa_2]}S_1$
%\end{example}
%\begin{example}
%\red{This is not the right example to clarify the above sentence.} \blue{Find a different example.}   Consider the following second order SMART
%    \[\cR\colon 2S_1\ce{->}0\ce{->}S_1\]
%    Note that the following first order endotactic SMART
%    \[\widehat{\cR}\colon 0\ce{<=>}S_1\]
%    is structurally equivalent to $\cR$. However, $\widehat{\cR}\not\subseteq\cR$ and hence the assumption (\textbf{SUB}) is \emph{not} fulfilled. Indeed, $\cR$ is a one-species second order endotactic SMART, and hence is exponentially ergodic (on the unique closed communicating class $\mathbb{N}_0$) regardless of the reaction rate constants \cite{H18} (see also \cite{XHW23} and \cite[Theorem~4.8\ and the remark that follows]{WX20}). 
%\end{example}

\begin{remark}
    Since any first order endotactic/strongly endotactic/weakly reversible SMART $\cR$ trivially satisfies (\textbf{SUB})  with $\widehat{\cR}=\cR$, Theorem~\ref{thm:exponential-ergodic-RN} generalizes Theorem~\ref{thm:exp-ergodicity-endotactic}.
\end{remark}

\section{Examples}\label{sect:examples}

Below we illustrate by diverse examples the wide applicability of  Theorem~\ref{thm:exponential-ergodic-RN} and Corollary~\ref{cor:exponential-ergodicity-for-nonlinear-models}.

\begin{example}\label{ex:virtual source}
Let $\mathcal{R}_1$ be a first order endotactic SMART consisting of species $S_1,\ S_2$ such that $\cR_1=(\cR_1)^0$. Let $\mathcal{R}=\mathcal{R}_1\cup\mathcal{R}_2$ where  $$\mathcal{R}_2=\{S_1+2S_2\ce{->[1]}4S_2,\ S_1+2S_2\ce{->[1]}3S_1+S_2,\ S_1+2S_2\ce{->[1]}S_2\}$$ It is easy to verify that neither $\cR$ nor $\cR_2$ is endotactic. Let $\widehat{\cR}=\cR_1$ and $v_0\in\mathbb{R}^2_{++}$ be such that $A_0v_0<0$, where $A_0$ is the net flow matrix of $\cR_1$. Note that $\Gamma=\mathbb{N}^2_0$ is the unique closed communicating class for both $\cR$ and $\widehat{\cR}$; moreover, \begin{align*}
    \sum_{y\ce{->}y'\in\cR\setminus(\widehat{\cR})^0}\lambda_{y\to y'}(x)(v_0\cdot(y'-y)) = & \sum_{y\ce{->}y'\in\cR_2}\lambda_{y\to y'}(x)(v_0\cdot(y'-y))\\
    = & x_1x_2(x_2-1)v_0\cdot ((-1,2)+(2,-1)+(-1,-1)) = 0 
\end{align*} Applying Theorem~\ref{thm:exponential-ergodic-RN} with $f\equiv0$, $\cR$ is exponentially ergodic on $\Gamma$. 
\end{example}
\begin{remark}
    As a side note, $S_1+2S_2$ in Example~\ref{ex:virtual source} is usually referred to as a \emph{virtual source} \cite[Definition~4.1]{CJY20} or \emph{ghost vertex} \cite{D23} whose outflows average out (i.e., reaction vectors from that vertex sum up to zero). Virtual sources are frequently used in problems on the topic of \emph{confoundability} of deterministic reaction systems \cite{CP08}.  
\end{remark}

%\begin{example}
%Consider the following second order SMART \[\cR\colon S_1\ce{->[$\kappa_1$]}S_2\ce{->[$\kappa_2$]}0\ce{->[$\kappa_0$]}S_1+S_2\quad 2S_3\ce{->[$\kappa_{30}$]}0\ce{->[$\kappa_3$]}S_3\ce{<=>[$\kappa_{34}$][$\kappa_{43}$]}S_4\] 
%It follows from Proposition~\ref{prop:endotactic=weakly-reversible} that $\cR$ is \emph{not} endotactic since its monomerization is not weakly reversible. Nevertheless, it 
%$\cR$ has a sub reaction network which is a first order endotactic SMART
%\[\widehat{\cR}\colon S_1\ce{->[$\kappa_1$]}S_2\ce{->[$\kappa_2$]}0\ce{->[$\kappa_0$]}S_1+S_2\quad S_3\ce{<=>[$\kappa_{34}$][$\kappa_{43}$]}S_4\] 
%Note that $\cR\setminus\widehat{\cR}=\{2S_3\ce{->[$\kappa_{30}$]}0\ce{->[$\kappa_3$]}S_3\}$. One can verify that and  \eqref{eq:condition-2} holds with $v=(2,1,1,1)$ and $C=\kappa_0$. Hence 
%\end{example}

\begin{example}\label{ex:adding-directional-decreasing-reactions}
Consider the following first order SMART:
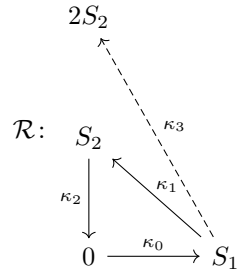
\[\mathcal{R}\colon 
\begin{tikzcd}[column sep=3em,row sep=3em]
2S_2&\\
S_2 \arrow[d,shift left=0ex,swap,"\kappa_2"]&\\
0\arrow[r,shift left=0.2ex,"\kappa_0"]&S_1 \arrow[ul,shift left=0ex,swap,"\!\!\kappa_1"]  \arrow[luu,shift left=0ex, dashed, swap, "\kappa_3"]
\end{tikzcd}
\] It follows from Proposition~\ref{prop:endotactic=weakly-reversible} that $\cR$ is not endotactic. Nevertheless, removing the reaction in the dashed arrow, we obtain a first order endotactic SMART $\widehat{\cR}\subseteq\cR$. By Theorem~\ref{thm:essential}, $\mathbb{N}^2_0$ is the unique closed communicating class for $\widehat{\mathcal{R}}^0$, and hence so is for $\mathcal{R}$. Observe that  $\cR\setminus(\widehat{\cR})^0 = \{S_1\ce{->[$\kappa_3$]}2S_2\}$. Choose $v=(2,1)\in\mathsf{S}_{\cR\setminus(\widehat{\cR})^0}^{\perp}$ so that  \eqref{eq:condition-3} is fulfilled. Then it follows from Corollary~\ref{cor:exponential-ergodicity-for-nonlinear-models}(ii) that $\cR$ is exponentially ergodic on $\mathbb{N}^2_0$.
\end{example}

\begin{example}
Consider the following second order SMART \[\cR\colon S_1\ce{->[$\kappa_1$]}S_2\ce{->[$\kappa_2$]}0\ce{->[$\kappa_0$]}S_1+S_2\quad S_3\ce{<=>[$\kappa_3$][$\kappa_4$]}S_4\quad S_2+S_3\ce{->[$\kappa_{23}$]}0\]    
 which is \emph{not} essential: It has  open communicating classes $\mathbb{N}^2_0\times\{(n,m-n)\colon n=0,\ldots,m\}$ for $m\in\mathbb{N}$ and a unique closed communicating class    $\Gamma\coloneqq\mathbb{N}^2_0\times\{(0,0)\}$. Nevertheless, it has a first order endotactic SMART $$\widehat{\cR}\colon S_1\ce{->[$\kappa_1$]}S_2\ce{->[$\kappa_2$]}0\ce{->[$\kappa_0$]}S_1+S_2\quad S_3\ce{<=>[$\kappa_3$][$\kappa_4$]}S_4$$ and $\Gamma$ is also a closed communicating class for $\widehat{\cR}$. Hence $$\cR\setminus(\widehat{\cR})^0 = \{S_3\ce{<=>[$\kappa_3$][$\kappa_4$]}S_4\quad S_2+S_3\ce{->[$\kappa_{23}$]}0\}$$ Applying Corollary~\ref{cor:exponential-ergodicity-for-nonlinear-models}(ii) with $v=(2,1,1,1)$, $\cR$ is exponentially ergodic on $\Gamma$.
\end{example}

\begin{example}\label{ex:adding-decreasing-reactions}
Consider the following weakly reversible SMART:
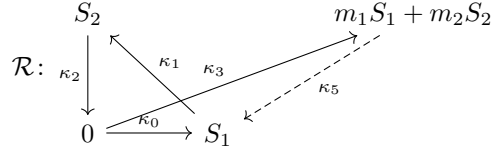
\[\mathcal{R}\colon 
\begin{tikzcd}[column sep=3em,row sep=3em]
S_2 \arrow[d,shift left=0ex,swap,"\kappa_2"]&&m_1S_1 + m_2S_2\arrow[dl,dashed,shift left=0ex,"\kappa_5"]\\
0\arrow[r,shift left=0.2ex,"\kappa_0"]\arrow[urr,shift left=0ex,"\kappa_3"]&S_1 \arrow[ul,shift left=0ex,swap,"\kappa_1"]  &
\end{tikzcd}
\]
where $m_1,\ m_2\in\mathbb{N}$. Removing the  decreasing reaction labeled by a dashed arrow, we obtain a first order endotactic  SMART $\widehat{\mathcal{R}}\subseteq \cR$ by Proposition~\ref{prop:endotactic=weakly-reversible}. %By Theorem~\ref{thm:essential} and Corollary~\ref{cor:exp-ergodicity-endotactic}(i), we know that $\mathcal{R}^0$ is essential  with a unique closed communicating class $\mathbb{N}^2_0$ and is exponentially ergodic,  %$$\widetilde{\mathcal{R}} = \{0\ce{<=>}S_1\ce{->}S_2\ce{->}0\ce{->}m_1S_1+m_2S_2\}$$
Similar to the previous example, one can show that $\mathbb{N}^d_0$ is the unique closed communicating class for both $\widehat{\mathcal{R}}$ and $\mathcal{R}$. %Note that $A=\begin{bmatrix}-\kappa_1 - \kappa_4&\kappa_1\\
%0&-\kappa_2\end{bmatrix}$. %and $b=[\kappa_0+\kappa_3m_1,\kappa_3m_2]$. 
%Then $A\mathbf{1}^T<0$. Moreover, 
%\begin{align*}
%\sum_{y\ce{->}y'\in\cR\setminus\cR^0}\lambda_{y\to y'}(x)\mathbf{1}^T(y'-y) = \kappa_5 x_1^{\underline{m_1}}x_2^{\underline{m_2}}(1- m_1- m_2)\le0
%\end{align*}
%\begin{align*}
%\mathcal{L}_{\mathcal{R}}v_{\varepsilon}^Tx = & v_{\varepsilon}^TAx+v_{\varepsilon}^T\Bigl(\kappa_0[1,0]^T+\kappa_3[m_1,m_2]^T\Bigr) - \kappa_4 x_1^{\underline{m_1}}x_2^{\underline{m_2}}v_{\varepsilon}^T[m_1-1,m_2]\le v_{\varepsilon}^T(Ax+b)
%\end{align*} 
Since $(\widehat{\cR})^0=\widehat{\cR}$, by Corollary~\ref{cor:exponential-ergodicity-for-nonlinear-models}(iii), $\mathcal{R}$ is exponentially ergodic. %with a unique stationary distribution on $\mathbb{N}^2_0$.
%In this example, we add a  decreasing reaction $m_1S_1 + m_2S_2\ce{->[\kappa_4]}S_1$ to a first order endotactic SRS. By Corollary~\ref{cor:exponential-ergodicity-for-nonlinear-models}(iii), $\mathcal{R}$ is exponentially ergodic. %Indeed, adding $S_1\ce{->[\kappa_5]}0$ preserves endotacticity, and hence $\cR$ can be regarded as a first order endotactic SMART with an added  bimolecular decreasing reaction (which does not belong to a first order reaction network).
\end{example}

%, e.g., the RN $\mathcal{R}'$ in the above example.

%The following general case follows from Theorem~\ref{thm:exponential-ergodic-RN},  Theorem~\ref{thm:endotactic} and Proposition~\ref{prop:endotactic-Hurwitz}.

%The results below immediately follow from that both weakly reversible RNs and strongly endotactic RNs are endotactic \cite{GMS14}. 
%\begin{corollary}\label{cor:ergodicity-weakly-reversible}
%Every MAK SRS kinetically equivalent from a non-conservative weakly reversible  monomolecular MAK SRS is exponentially ergodic on every closed communicating class. Particularly, every first order weakly reversible MAK SRS is exponentially ergodic on every closed communicating class.
%\end{corollary}
%\begin{proof}
%Let $\mathcal{R}$ be a non-conservative endotactic  monomolecular MAK SRS. It follows from Proposition~\ref{prop:endotactic-Hurwitz} that the net flow matrix of $\mathcal{R}$ is Hurwitz. Then the conclusion follows from Theorem~\ref{thm:exponential-ergodic-RN}.
%\end{proof}

%For any given SRS $(\cR,\Lambda)$, the sub SRS $(\cR_1,\Lambda_1)$ with $\Lambda_1\coloneqq\{\lambda_{y\to y'}\in\Lambda\colon y\ce{->}y'\in\cR_1\}$ is called the \emph{linearization of $(\cR,\Lambda)$}.  

Next, we proceed to provide examples which builds upon first order endotactic SMART modules.

\begin{example}\label{ex:copies-exp-erg}
Consider the following third order SMART \[\mathcal{R}\colon 
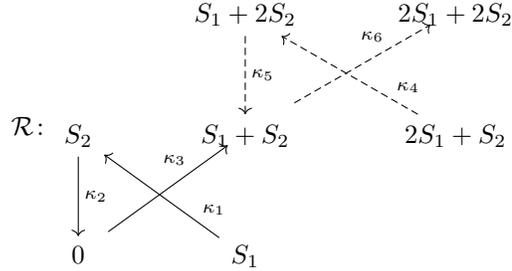
\begin{tikzcd}[column sep=3em,row sep=3em]
&S_1+2S_2 \arrow[d,shift left=0ex,dashed,"\kappa_5"] & 2S_1+2S_2\\
S_2 \arrow[d,shift left=0ex,"\kappa_2"] &S_1+S_2\arrow[ur,xshift =1ex,yshift =1ex, pos=.7, dashed, "\kappa_6"]& 2S_1+S_2\arrow[lu, pos=0.2, swap, dashed, "\kappa_4"]\\
0\arrow[ur,xshift =1ex,yshift =1ex, pos=.7, "\kappa_3"]&S_1\arrow[ul, pos=0.2, swap, "\kappa_1"]&
\end{tikzcd}
\] 
Note that the set of reactions in dashed arrows can be regarded as a translation (by $(1,1)$) of the set $\widehat{\cR}$ of remaining reactions in $\cR$ disrespecting the kinetic rate constants. It is readily verified that the underlying CTMC for either $\widehat{\cR}$ or $\cR$ is irreducible on $\Gamma=\mathbb{N}^2_0$. Moreover, it is also straightforward to see that $v_0=(2,1)$ satisfying $A_0v_0<0$ for the net flow matrix $A_0$ of $(\widehat{\cR})^0=\widehat{\cR}$, and % $$A_1=\begin{bmatrix}
%    -\kappa_1&\kappa_1\\
 %   0&-\kappa_2
%\end{bmatrix},\quad A_2\begin{bmatrix}
 %-\kappa_5&\kappa_5\\
  %  0&-\kappa_6
%\end{bmatrix},\quad b_1=(\kappa_3,\kappa_3)^T,\quad b_2=(\kappa_4,\kappa_4)^T$$ 
%Then  
\begin{align*}
\sum_{y\to y'\in\cR\setminus(\widehat{\cR})^0}\lambda_{y\to y'}(v\cdot(y'-y)) %= & x_1x_2 ((x-\mathbf{1})\cdot(A_2)v_0+b_2\cdot v_0) \\
= & -x_1x_2(\kappa_4(x_1-1)+\kappa_5(x_2-1)- 3\kappa_6)\le 0,
\end{align*}
for all $x\in\Gamma$ such that $\kappa_4(x_1-1)+\kappa_5(x_2-1)\ge 3\kappa_6$. Since \[\{(x_1,x_2)\colon \kappa_4(x_1-1)+\kappa_5(x_2-1)\boldsymbol{<} 3\kappa_6\}\cap\Gamma\] is finite, \eqref{eq:condition-1} is fulfilled. Applying Theorem~\ref{thm:exponential-ergodic-RN} yields that $\cR$ is exponentially ergodic.
%\begin{align*}
%    \mathcal{L}_{\cR}V(x) & = (x^TA_1v_0 +b_1^Tv_0) + x_1x_2((x-\mathbf{1})^TA_2v_0+b_2^Tv_0)\\
%    &\le -c_1V(x)-x_1x_2c_2V(x-\mathbf{1})\\
%    & \le -\min\{c_1,c_2\}V(x)
%\end{align*} for all $x$ of large $\ell_1$-norm, 
%where $v_0=1+\varepsilon$ for some $\varepsilon>0$, $c_,\ c_2>0$ depending on $\varepsilon$, and 
\end{example}

\begin{remark}
As a side note, despite that every \emph{one-dimensional} endotactic SMART is exponentially ergodic in each closed communicating class  \cite[Theorem~4.8]{WX20}, and that the joint of (translations of \emph{different}) one-dimensional endotactic reaction networks is still endotactic, such a joint SMART may \emph{not preserve exponential ergodicity} in general, e.g., Example~\ref{ex:sub-exponential-ergodicity}.
\end{remark}

Below we provide an example of \textbf{a SMART with a first order endotactic SMART \emph{asymptotic limit}}.%to illustrate how this definition comes into play.

\begin{example}\label{ex:bimolecular} Consider the following second order SMART
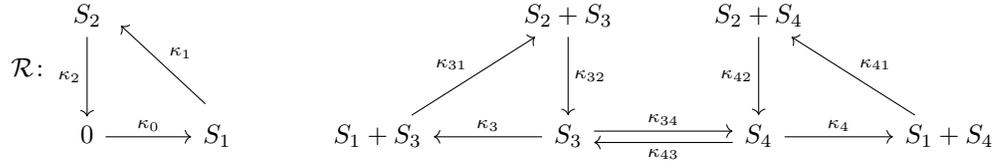
\[\mathcal{R}\colon 
\begin{tikzcd}[column sep=3em,row sep=3em]
S_2 \arrow[d,shift left=0ex,swap,"\kappa_2"] & &&S_2+S_3 \arrow[d,shift left=0ex,"\kappa_{32}"] & S_2+S_4 \arrow[d,shift left=0ex,swap,"\kappa_{42}"]\\
0 \arrow[r,shift left=0ex,"\kappa_0"] &S_1\arrow[ul,xshift =1ex,yshift =1ex, swap,"\kappa_1"]& S_1+S_3\arrow[ur, "\kappa_{31}"]&S_3\arrow[l,  swap, "\kappa_3"]\arrow[r, yshift =.5ex,  "\kappa_{34}"]&S_4\arrow[r, yshift =0ex,  "\kappa_4"]\arrow[l,  yshift =-.5ex, "\kappa_{43}"]&S_1+S_4\arrow[ul,  swap, "\kappa_{41}"]
\end{tikzcd}
\] 
Let $\cR_1=\{S_3\ce{<=>[\kappa_{34}][\kappa_{43}]}S_4\}$. Obviously, $\cR_1$ is a \emph{conservative} sub reaction network (with $(1,1)\in\mathsf{S}_{\cR_1}^{\perp}$) with its species $S_3$ and $S_4$ being catalysts in part of the reactions in  $\cR\setminus\cR_1$; moreover, the WRDZ SMART $\cR_1$ as modeled by a finite state irreducible CTMC is exponentially ergodic with a Poisson stationary distribution $\pi$ \cite[Theorem~4.2]{ACK10} over a finite set of states. Conditioned on $w=(w_3,w_4)\in\mathbb{N}^2_0$ counts of species $S_3$ and $S_4$, $\cR$ may be represented by the following \textbf{conditional reaction network}:
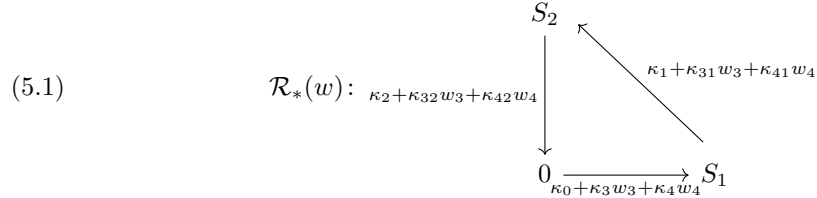
\begin{equation}\label{eq:conditional_network}
\mathcal{R}_*(w)\colon 
\begin{tikzcd}[column sep=4.5em,row sep=4.5em]
S_2 \arrow[d,shift left=0ex,swap,"\kappa_2+\kappa_{32}w_3+\kappa_{42}w_4"] \\
0 \arrow[r,shift left=0ex,swap,"\kappa_0+\kappa_3w_3+\kappa_4w_4"] &S_1\arrow[ul,xshift =1ex,yshift =1ex, swap,"\kappa_1+\kappa_{31}w_3+\kappa_{41}w_4"]
\end{tikzcd}
\end{equation}It is readily verified that $\cR_*(w)$ as a SMART is monomolecular and WRDZ, and hence by Theorem~\ref{thm:exp-ergodicity-endotactic} and \cite[Theorem~4.2]{ACK10} is exponentially ergodic with a Poisson stationary distribution. Observe that the evolution of the counts of species $S_3$ and $S_4$ is governed by $\cR_1$ and is \emph{independent} of that of the counts of $S_1$ and $S_2$. Therefore the marginal distribution of $\cR$ to species $S_3$ and $S_4$ converges to $\pi$ as time tends to infinity. Akin to and inspired by the limit of \emph{asymptotic autonomous ODEs} \cite{T92}, one can define  
 $\{\cR_*(w)\colon w\in\mathbb{N}^2_0\}$\---\emph{a family of conditional SRS of $\cR$}, as the \emph{asymptotic limit} of $\cR$. Indeed, using similar argument as in Example~\ref{ex:copies-exp-erg}, one can also apply Theorem~\ref{thm:exponential-ergodic-RN} to show that $\cR$ is exponentially ergodic. %one can verify that $V(x)=2x_1+x_2+x_3+x_4$ is a Lyapunov function (in the sense of Proposition~\ref{prop:drift_criterion_for_exp-ergodicity}) for exponential ergodicity of $\cR$. %one can 
\end{example}
\begin{remark}
\begin{itemize}
    \item It is noteworthy that in comparison, $\cR$ in Example~~\ref{ex:bimolecular} modeled as a deterministic mass-action system has a unique globally asymptotically stable equilibrium in each stoichiometric compatibility class \cite{X25a}. Indeed, as a deterministic reaction system of species $S_1$ and $S_2$ with kinetics depending on concentration of species $S_3$ and $S_4$, the deterministic mass-action system  confined to species $S_1$ and $S_2$ on the stoichiometric compatibility class with positive $c$ total concentration of species $S_3$ and $S_4$ is \emph{asymptotically autonomous with the limit} being the following deterministic mass-action system 
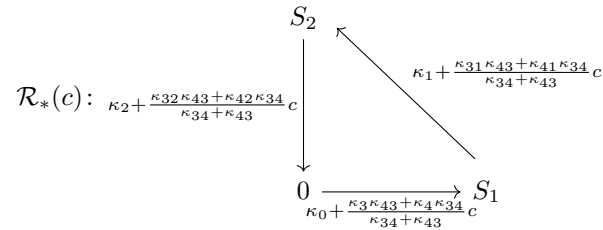
\[\mathcal{R}_*(c)\colon 
\begin{tikzcd}[column sep=5em,row sep=5em]
S_2 \arrow[d,shift left=0ex,swap,"\kappa_2+\frac{\kappa_{32}\kappa_{43}+\kappa_{42}\kappa_{34}}{\kappa_{34}+\kappa_{43}}c"] \\
0 \arrow[r,shift left=0ex,swap,"\kappa_0+ \frac{\kappa_{3}\kappa_{43}+\kappa_{4}\kappa_{34}}{\kappa_{34}+\kappa_{43}}c"] &S_1\arrow[ul,xshift =1ex,yshift =1ex, swap,"\kappa_1+\frac{\kappa_{31}\kappa_{43}+\kappa_{41}\kappa_{34}}{\kappa_{34}+\kappa_{43}}c"]
\end{tikzcd}
\]
\item Example~\ref{ex:bimolecular} provides a little further insight into the study of stability of \emph{randomly switching SRS} \cite{CHX25}. Indeed, from the perspective of randomly switching reaction networks, the reversible reactions 
$S_3\ce{<=>[\kappa_{34}][\kappa_{43}]}S_4$  act as the \emph{environment} (in other words, $S_3$ and $S_4$ together act as a ``switch'' by confining their total molecule counts to be 1).
\end{itemize}
\end{remark}

In Example~\ref{ex:bimolecular}, despite the sub reaction system governing the evolution of the counts of the catalysts $S_3$ and $S_4$ is conservative, \emph{the conservativity of the sub reaction system of catalytic species is unnecessary} for exponential ergodicity of a SMART with a first order endotactic SMART asymptotic limit. Indeed, Theorem~\ref{thm:exponential-ergodic-RN} does \emph{not} apply while the similar argument in that proof still suffices, which makes it appear promising to push further the arguments for a stronger result that we may pursue in the future.

\begin{example}\label{ex:open}
    Consider the following SMART as a variant of Example~\ref{ex:bimolecular}:
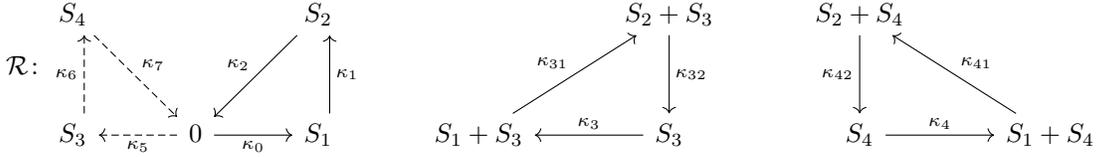
    \[\mathcal{R}\colon 
\begin{tikzcd}[column sep=3em,row sep=3em]
S_4\arrow[dr,shift left=0ex, dashed, "\kappa_7"]&&S_2 \arrow[dl,shift left=0ex,swap,"\kappa_2"] & &S_2+S_3 \arrow[d,shift left=0ex,"\kappa_{32}"] & S_2+S_4 \arrow[d,shift left=0ex,swap,"\kappa_{42}"]&\\
S_3\arrow[u,xshift =1ex, dashed, "\kappa_6"]&0 \arrow[l,shift left=0ex, dashed, "\kappa_5"]\arrow[r,shift left=0ex, swap, "\kappa_0"] &S_1\arrow[u,xshift =1ex,swap,"\kappa_1"]& S_1+S_3\arrow[ur, "\kappa_{31}"]&S_3\arrow[l, swap, "\kappa_3"]&S_4\arrow[r, yshift =0ex,  "\kappa_4"]&S_1+S_4\arrow[ul,  swap, "\kappa_{41}"]
\end{tikzcd}
\] 
It is readily verified that the underlying CTMC on $\Gamma=\mathbb{N}^4_0$ is irreducible. Despite similar to Example~\ref{ex:bimolecular}, the evolution of $S_3$ and $S_4$ is independent from the evolution of $S_1$ and $S_2$, it is governed by an \emph{open}\footnote{To account for the case where mass exchange with the ambient is possible, an open reaction system in the sense of Feinberg \cite{F87} refers to one that contains ``pseudo-reactions''\---those contain the zero complex as either the reactant or the product.} first order endotactic SMART\---the sub SRS consisting of reactions in dashed arrows. %5,6,7 (in terms of the indices of the respective reaction rate constants). 
Moreover, $\cR$ has the same asymptotic limit $\{\cR_*(w)\colon w\in\mathbb{N}^2_0\}$ as defined by the reaction network \eqref{eq:conditional_network}. 

Note that reaction graph of $\cR$ is composed of three strongly connected components. To apply Theorem~\ref{thm:exponential-ergodic-RN}, in order for $\Gamma$ to be a closed communicating class also for $\widehat{\cR}$, based on the characterization of first order endotactic reaction networks from Proposition~\ref{prop:endotactic=weakly-reversible}, we can only choose $\widehat{\cR}$ to be the strongly connected component of $\cR$ on the left (rather than choose one of the two triangles in that component). Without much effort, one can verify that there exists no $v_0\in\mathbb{R}^4_{++}$ fulfilling \eqref{eq:condition-1} and hence Theorem~\ref{thm:exp-ergodicity-endotactic} is \emph{not} applicable. Nevertheless, there still exists a linear Lyapunov function for exponential ergodicity. Let $v=(2,1,2c,c)$ with a positive constant $c>\max\{\tfrac{\kappa_3}{\kappa_6},\tfrac{\kappa_4}{\kappa_7}\}$, then
\begin{align*}
    \mathcal{L}_{\cR} (v\cdot x) = &\sum_{y\to y'\in\cR}\lambda_{y\to y'}(x) v\cdot(y'-y)\\
   % =& \kappa_5v_3-\kappa_6(v_3-v_4)x_3-\kappa_7v_4x_4+v_1(\kappa_0+\kappa_3x_3+\kappa_4x_4)\\
    &-(v_1-v_2)x_1(\kappa_1+\kappa_{31}x_3+\kappa_{41}x_4)-v_2x_2(\kappa_2+\kappa_{32}x_3+\kappa_{42}x_4)\\
   = & 2c\kappa_5+\kappa_0-(c\kappa_6-\kappa_3)x_3-(c\kappa_7-\kappa_4)x_4\\
    &-x_1(\kappa_1+\kappa_{31}x_3+\kappa_{41}x_4)-x_2(\kappa_2+\kappa_{32}x_3+\kappa_{42}x_4)\\
    \le&2c\kappa_5+\kappa_0-C (x\cdot v),
\end{align*}
where $C=\min\{\tfrac{\kappa_1}{2},\kappa_2,\tfrac{\kappa_6}{2}-\tfrac{\kappa_3}{2c},\kappa_7-\tfrac{\kappa_4}{c}\}>0$. This yields the exponential ergodicity of $\cR$ by Proposition~\ref{prop:drift_criterion_for_exp-ergodicity}.
\end{example}

One might tend to believe that Theorem~\ref{thm:exponential-ergodic-RN} applies to or a linear Lyapunov function exists for Example~\ref{ex:copies-exp-erg}, Example~\ref{ex:bimolecular}, or Example~\ref{ex:open}, is a \emph{coincidence} with the fact that in these examples higher order reactions constitute (in part) translations of a same first order endotactic reaction network as a building block. The following example shows that Theorem~\ref{thm:exponential-ergodic-RN} is \emph{not} limited to these patterns.

%In the above example, in each $e_i$-direction, there is a dominating reaction which is not necessarily of first order but is decreasing along $v_0$, a common positive direction. 

%Now we vary the above example a little, and interestingly, Theorem~\ref{thm:exponential-ergodic-RN} \emph{applies}. 

    \begin{example}\label{ex:3copies-exp-ergodic}
Consider the following third order SMART \[\mathcal{R}\colon 
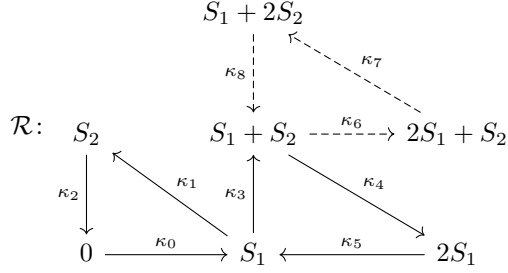
\begin{tikzcd}[column sep=3em,row sep=3em]
&S_1+2S_2 \arrow[d,shift left=0ex,swap,"\kappa_8",dashed] & \\
S_2 \arrow[d,shift left=0ex,swap,"\kappa_2"] &S_1+S_2\arrow[r,shift left=0ex,"\kappa_6",dashed]\arrow[dr,shift left=0ex,"\kappa_4"]& 2S_1+S_2\arrow[lu,shift left=0ex,swap,"\kappa_7",dashed]\\
0\arrow[r,shift left=0ex,"\kappa_0"]&S_1\arrow[ul,shift left=0ex,swap,"\kappa_1"]\arrow[u,shift left=0ex,"\kappa_3"]&2S_1\arrow[l,shift left=0ex,"\kappa_5",swap]
\end{tikzcd}
\]  Let $\cR_1$ and $\cR_2$ be the lower left and lower right triangular sub reaction networks in $\cR$, respectively. Observe that $\cR_2$ is \emph{not} a translation of $\cR_1$, but a translation of the \emph{reverse} of $\cR_1$ (in the sense that all reactions in $\cR_1$ are reversed). Moreover, $\cR$ and $\cR_1$ share the unique  communicating class $\Gamma=\mathbb{N}^2_0$ which is closed. Choose $\widehat{\cR}=\cR_1$, and let $v=(2,1)$ such that $A_1v<0$, where $A_1$ is the net flow matrix of $\cR_1$. Straightforward calculations yield
\begin{align*}
   \sum_{y\to y'\in\cR\setminus(\widehat{\cR})^0}\lambda_{y\to y'}(v\cdot(y'-y)) = & x_1g(x),
\end{align*}
where $$g(x)=\kappa_3+\kappa_4x_2-2(x_1-1)\kappa_5 +x_2(2\kappa_6-\kappa_7(x_1-1)-\kappa_8(x_2-1))$$ Since $\{x\colon g(x)>0\}\cap\mathbb{R}^2_+$ is bounded, we have $\{x\colon g(x)>0\}\cap\mathbb{N}^2_0$ is finite. Hence \eqref{eq:condition-1} is fulfilled with $f\equiv0$ and it follows from Theorem~\ref{thm:exponential-ergodic-RN} that $\cR$ is exponentially ergodic.  
\end{example}
%Different from the above examples, the example below is \emph{not} weakly reversible.

\begin{example}\label{ex:swap-2RNs}
Removing the set of reactions in dashed arrows from the SMART in Example~\ref{ex:3copies-exp-ergodic}, we obtain a second order SMART: $\cR=\cR_1\cup\cR_2$, where $\cR_1$ and $\cR_2$ are defined in Example~\ref{ex:3copies-exp-ergodic}. Since $\cR_1\subseteq\cR$, $\mathbb{N}^2_0$ is again the unique closed communicating class for $\cR$. %; moreover, $\cR$ is the joint of translations of the following two first order weakly reversible reaction networks: 
%\[\mathcal{R}_1\colon 
%\begin{tikzcd}[column sep=3em,row sep=3em]
%S_2 \arrow[d,shift left=0ex,swap,"\kappa_2"] &\\
%0\arrow[r,shift left=0ex,"\kappa_0"]&S_1\arrow[ul,shift left=0ex,swap,"\kappa_1"]&
%\end{tikzcd}\quad \mathcal{R}_2\colon 
%\begin{tikzcd}[column sep=3em,row sep=3em]
%S_2 \arrow[dr,shift left=0ex,"\kappa_4"] &\\
%0\arrow[u,shift left=0ex,"\kappa_3"]&S_1\arrow[l,shift left=0ex,swap,"\kappa_5"]&
%\end{tikzcd}
%\]
The net flow matrices of $\cR_1$ and $\cR_2$ are given by \begin{equation}\label{eq:A_1}
    A_1=\begin{bmatrix}
    -\kappa_1&\kappa_1\\
    0&-\kappa_2
\end{bmatrix},\quad A_2\begin{bmatrix}
-\kappa_5&0\\
    \kappa_4&-\kappa_4
\end{bmatrix}
\end{equation}Observe that \emph{$A_1$ and $A_2$ do not share a common decreasing direction}: For any positive vector $v>0$ satisfying $A_1v<0$, we have $A_2v\not\le0$. Because of this observation, it is straightforward to check that Theorem~\ref{thm:exponential-ergodic-RN} fails to apply to $\cR$. Nevertheless, $\mathbf{1}\cdot x$ is a Lyapunov function for exponentially ergodicity of $\cR$, since %$$A_i\mathbf{1}\le 0\quad \text{ for } i=1,2;\quad (A_1+A_2)\mathbf{1}\le -\min\{\kappa_2,\kappa_5\}\mathbf{1}$$ %However, the condition \eqref{eq:strictly_negative} in Theorem~\ref{thm:exponential_ergodicity_of_joint_of_translations} is fulfilled with $v=\mathbf{1}$. Indeed, let  $V(x)=x\mathbf{1}^T$. Then
%Then
\begin{align*}
    \mathcal{L}_{\cR}(\mathbf{1}\cdot x)  = &  -\kappa_5x_1(x_1-1)+\kappa_3x_1-\kappa_2x_2+\kappa_0\\
     \le & -\kappa_2 (x_1+x_2)+\frac{(\kappa_2+\kappa_3+\kappa_5)^2}{4\kappa_5}+\kappa_0 \le -\frac{\kappa_2}{2}\mathbf{1}\cdot x,
\end{align*}
for all but finitely many $x\in\Gamma$.%, and hence $\cR$ is  by Proposition~\ref{prop:drift_criterion_for_exp-ergodicity}. %It is noteworthy that $\mathbf{1} \cdot x$ is \emph{not} a Lyapunov function for exponential ergodicity of the SRS in Example~\ref{ex:copies-exp-ergodic}. 

%Let $v=(1,1+\epsilon)$ for some positive $\epsilon$. Then $A_2v<0$ and 
%\begin{align}
 %   \mathcal{L}_{\cR}v \cdot x = & -\kappa_4x_1x_2\epsilon-\kappa_5x_1(x_1-1)+\kappa_3x_1(1+\epsilon) +\kappa_1x_1\epsilon -\kappa_2x_2(1+\epsilon) +\kappa_0\\
  %  \le& - \kappa_2(1+\epsilon)x_2-\kappa_5(x_1-\tfrac{\kappa_5+\kappa_3(1+\epsilon)+\kappa_1\epsilon}{2\kappa_5})^2 + \tfrac{(\kappa_5+\kappa_3(1+\epsilon)+\kappa_1\epsilon)^2}{4\kappa_5}+\kappa_0,
%\end{align}
%which is also a Lypunov function for exponential ergodicity. Indeed, $v\cdot x$ ensures negative coefficients of all leading terms  $x_1^2$, $x_1x_2$ and $x_2$ in  $\mathcal{L}_{\cR} v\cdot x$.
\end{example}

\section{Discussion and Outlooks}\label{sect:discussion}

Extensive generalization of these examples in Section~\ref{sect:nonlinear} is left for a future work. A careful reader may notice that we \emph{lack} an example illustrating the case of Corollary~\ref{cor:exponential-ergodicity-for-nonlinear-models}(i). An SRS with \emph{non} mass action kinetics might be easily constructed as an example, however, a SMART example seems nontrivial to the author.

Based on Theorem~\ref{thm:exp-ergodicity-endotactic}, \cite[Lemma~4.15]{X25a}, as well as \cite[Theorem~4, Theorem~8]{CHX25}, one can show that \textbf{a Markov process that \emph{randomly switches among finitely many first order endotactic SMART} is exponentially ergodic for both small and large switching rates} \cite{X26}; furthermore, different \emph{types} of examples \cite{X26} suggest that \emph{such exponential ergodicity result holds regardless of the switching rates}. It is noteworthy that such a class of randomly switching first order SMART can be represented as second order SMART, with the added \emph{environmental} species that independently evolve as a sub conservative weakly reversible SMART as switches that catalyze other first order reactions. Hence understanding stability of such a class of SMART may also improve our understanding of Conjecture~\ref{conj:positive-recurrence} for bimolecular SMART.

\appendix

\section{Structural embedding of a first order reaction network into its monomerization}

\begin{proposition}\label{prop:embedding}
    Let $\cR$ be a first order reaction network and $\cR^{\spadesuit}$ be its monomerization. Assume $\cR$ and $\cR^{\spadesuit}$ as SRS are both generic. Then $\cR\hookrightarrow\cR^{\spadesuit}$.
\end{proposition}

\begin{proof}
    Let $x\rightharpoonup z$ for $\mathcal{R}$. Then there exists a sequence of (possibly repeated) reactions $y_i\ce{->} y_i'\in\mathcal{R}$ for $i=1,\ldots,m$ such that $$x+\sum_{j=1}^{i-1}(y_j'-y_j)\ge y_i,\quad i=1,\ldots,m$$
where $\sum_{j=1}^0(y_j'-y_j)=0$ by convention, and $x+\sum_{j=1}^{m-1}(y_j'-y_j)=z$. We label these reactions by the index. Assume w.l.o.g. that $y_i=0$ for some $1\le i\le m$. Otherwise, by the definition of monomerization, $y_i\ce{->} y_i'\in\mathcal{R}^{\spadesuit}$ for all $i=1,\ldots,m$ and hence $x\rightharpoonup z$ for $\cR^{\spadesuit}$. We further assume the $i$-th reaction is the first one with a zero reactant (so that $y_j\neq0$ for $j<i$). It suffices to show that $x+\sum_{j=1}^{i-1}(y_j'-y_j)\rightharpoonup x+\sum_{j=1}^{i}(y_j'-y_j)$ for $\mathcal{R}^{\spadesuit}$. Note that $x+\sum_{j=1}^{i}(y_j'-y_j)=x+\sum_{j=1}^{i-1}(y_j'-y_j)+y_i'$. By the definition of $\mathcal{R}^{\spadesuit}$, $0\ce{->}S_k\in\mathcal{R}^{\spadesuit}$ for each $k\in\supp y_i'$. Hence the state $x+\sum_{j=1}^{i}(y_j'-y_j)$ can be reached from the state $x+\sum_{j=1}^{i-1}(y_j'-y_j)$ via repetitions of reactions of $0\ce{->}S_k$ for $(y_i')_k$ times for all $k\in\supp y_i'$. Indeed, all reactions $0\ce{->}S_k\in\mathcal{R}^{\spadesuit}$ are active on any intermediate state connecting the state $x+\sum_{j=1}^{i-1}(y_j'-y_j)$ leads to the state $x+\sum_{j=1}^{i}(y_j'-y_j)$ since any state in the state space is no smaller than $0$ component-wise. Now based on the same argument one can show by induction that $x+\sum_{j=1}^{i-1}(y_j'-y_j)$ to the state $x+\sum_{j=1}^{i}(y_j'-y_j)$ for the cases  where $i$ is not the first reaction with a zero reactant. This shows %if $x\rightharpoonup z$ for $\cR$, then 
$x\rightharpoonup z$ for $\cR^{\spadesuit}$. 
\end{proof}

\section{Lyapunov-Foster-Lyapunov criterion for (exponential) ergodicity}
A non-negative function $V$ defined on a unbounded subset of $\mathbb{R}^d$ %$V\colon\mathcal{X}\to\mathbb{R}_+$
 is \emph{norm-like} \cite{MT93} if $\lim\limits_{\|x\|_1\to\infty}V(x)=\infty$. 
\begin{proposition}\cite[Theorem~6.1]{MT93}\label{prop:drift_criterion_for_exp-ergodicity}
Let $X_t$ be an irreducible CTMC on the state space $\Gamma \subseteq \mathbb{R}^d$ and $\cL$ be its extended  generator. Then 
\begin{itemize}
    \item 
$X_t$ is ergodic on $\Gamma$ if there exists a positive constant $C$ and a non-negative norm-like function $V$ such that \[\mathcal{L}V(x)\le-C\] for all but finitely many states $x\in\Gamma$; 
\item $X_t$ is exponentially ergodic on $\Gamma$ if  there exists a positive constant $C$ and a non-negative norm-like function $V$ such that \[\mathcal{L}V(x)\le-C V(x)\] for all but finitely many states $x\in\Gamma$. 
\end{itemize}
 \end{proposition}

\bibliographystyle{plain}
{\small\bibliography{references}}

\end{document}